\documentclass{amsart}
\usepackage{amsmath,amssymb,amsfonts}
\usepackage{mathtools}
\usepackage{esvect}
\usepackage[shortlabels]{enumitem}
\setlist{leftmargin=*}

\usepackage{amsfonts}
\usepackage[colorlinks=true, linkcolor=blue]{hyperref}
\usepackage[sorted]{amsrefs}

\newtheorem{thm}{Theorem}[section] 
\newtheorem{cor}[thm]{Corollary}
\newtheorem{lem}[thm]{Lemma} 
\newtheorem{prop}[thm]{Proposition}
\newtheorem{claim}[thm]{Claim} 
\newtheorem{fact}[thm]{Fact}
\theoremstyle{definition} 
\newtheorem{defn}[thm]{Definition}
\theoremstyle{remark} 

\newtheorem{rem}[thm]{Remark}
\newtheorem{conj}[thm]{Conjecture}
\newtheorem{prob}[thm]{Problem}

\newtheorem{sample}[thm]{Example} \numberwithin{equation}{section}
    {\medskip\begingroup\leftskip 0.5cm\rightskip 0.5cm\noindent\begin{small}{\bf Remark.}}
    {\end{small}\par\endgroup}
{\begin{list}{$\bullet$}
 {\settowidth{\labelwidth}{\textsf{$\bullet$}} \setlength{\leftmargin}{10pt}}}
{\end{list}}

\newcounter{ssample}[section]

{\color{Bittersweet} \noindent Example \refstepcounter{ssample}\hbox{\bf \arabic{section}.\arabic{ssample}.}}
{}

\newcounter{insertcount}

    {\color{blue} \medskip\begingroup\noindent\begin{small}{\color{blue} \stepcounter{insertcount}
          {
            \bf Insert \arabic{insertcount}. #1.}
            \addcontentsline{toc}{subsection}{{\ \ \small  Insert \arabic{insertcount}: #1}}
               \leavevmode  } 
           }
    {\end{small}\par\endgroup}

\newcommand{\mrmk}[1]
{{\tiny$^{\spadesuit}$}\marginpar{\fbox{\footnotesize #1}}}
\def\strutdepth{\dp\strutbox}%
\def\marginalnote#1{\strut\vadjust{\kern-\strutdepth\specialnote{#1}}}%
\def\specialnote#1{\vtop to \strutdepth{\baselineskip%
\strutdepth\vss\llap{\hbox{\scriptsize \bf #1}}\null}}%

\newcommand{\RR}{\mathbb{R}}

\newcommand{\UU}{\mathbb{M}}

\def\op{\circ} 

\def\Aut{\operatorname{Aut}}

\def\twr{\operatorname{twr}}

\def\rd{\operatorname{rd}}

\def\Ind{\setbox0=\hbox{$x$}\kern\wd0\hbox to 0pt{\hss$\mid$\hss} \lower.9\ht0\hbox to 0pt{\hss$\smile$\hss}\kern\wd0} 

\def\Notind{\setbox0=\hbox{$x$}\kern\wd0\hbox to 0pt{\mathchardef \nn=12854\hss$\nn$\kern1.4\wd0\hss}\hbox to 0pt{\hss$\mid$\hss}\lower.9\ht0 \hbox to 0pt{\hss$\smile$\hss}\kern\wd0}

\newcommand{\QQ}{\mathbb{Q}}

\def\CF{\mathcal F}
\newcommand{\ZZ}{\mathbb{Z}}
\newcommand{\CR}{\mathcal R} 
\newcommand{\NN}{\mathbb N}

\newcommand{\la}{\langle}
\newcommand{\ra}{\rangle}

\def\tp{\mathrm{tp}}
\def\dcl{\mathrm{dcl}}

\reversemarginpar

\def\CL{\mathcal{L}}

\def\CG{\mathcal G}

\gdef\CM{\mathcal{M}}

\def\CM{\mathcal{M}}

\def\Th{Th}

\def\UU{\mathbb{M}}

\newcommand*\bbar[1]{%
  \hbox{%
    \vbox{%
      \hrule height 0.5pt 
      \kern0.5ex
      \hbox{%
        \kern-0.1em
        \ensuremath{#1}%
        \kern-0.1em
      }%
    }%
  }%
}

\gdef\CK{\mathcal{K}}

\gdef\Rexp{\RR_{\mathrm{exp}}}

\def\CM{\mathcal{M}}

\def\UU{\mathbb{M}}

\gdef\CM{\mathcal{M}}

\gdef\twr{\operatorname{twr}}
\gdef\op{\operatorname{op}}

\makeatletter
\DeclareRobustCommand{\cev}[1]{%
  \mathpalette\do@cev{#1}%
}
\newcommand{\do@cev}[2]{%
  \fix@cev{#1}{+}%
  \reflectbox{$\m@th#1\vec{\reflectbox{$\fix@cev{#1}{-}\m@th#1#2\fix@cev{#1}{+}$}}$}%
  \fix@cev{#1}{-}%
}
\newcommand{\fix@cev}[2]{%
  \ifx#1\displaystyle
    \mkern#23mu
  \else
    \ifx#1\textstyle
      \mkern#23mu
    \else
      \ifx#1\scriptstyle
        \mkern#22mu
      \else
        \mkern#22mu
      \fi
    \fi
  \fi
}

\title{Ramsey growth in some NIP structures}

\author[Chernikov]{Artem Chernikov} \address{Department of Mathematics, 
University of California Los Angeles,
Los Angeles, CA 90095-1555} \email{chernikov@math.ucla.edu}

\author[Starchenko]{Sergei Starchenko} \address{Department of Mathematics, University of Notre Dame, Notre Dame,
  IN 46556} \email{Starchenko.1@nd.edu}
  
\author[Thomas]{Margaret E. M. Thomas} \address{Zukunftskolleg, Department of Mathematics and Statistics, University of Konstanz, Box 216, 78457 Konstanz,
Germany} \email{margaret.thomas@uni-konstanz.de}

\begin{document}

{\abstract We investigate bounds in Ramsey's theorem for relations definable in NIP structures. Applying model-theoretic methods to finitary combinatorics, we generalize a theorem of Bukh and Matousek \cite{bukh2014erdHos} from the semialgebraic case to arbitrary polynomially bounded $o$-minimal expansions of $\mathbb{R}$, and show that it doesn't hold in $\mathbb{R}_{\exp}$. This provides a new combinatorial characterization of polynomial boundedness for $o$-minimal structures. We also prove an analog for relations definable in $P$-minimal structures, in particular for the field of the $p$-adics. Generalizing \cite{conlon2014ramsey}, we show that in distal structures the upper bound for $k$-ary definable relations is given by the exponential tower of height $k-1$.}


\maketitle

\section{Introduction}

We recall a fundamental theorem of Ramsey. Let $X$ be a set and let $E \subseteq X^k$  be a $k$-ary  relation on $X$. We say that a sequence $(a_i : 1\leq i \leq m)$ of elements in $X$ with $m \geq k$ is \emph{$E$-indiscernible} (also called ``\emph{$E$-homogeneous}'' in the literature) if either $E$ holds on all $k$-tuples $(a_{i_1}, \ldots, a_{i_k})$ with $1 \leq i_1 < \ldots < i_k \leq m$, or $E$ doesn't hold on any $k$-tuple $(a_{i_1}, \ldots, a_{i_k})$ with $1 \leq i_1 < \ldots < i_k \leq m$.

\begin{fact}[Ramsey \cite{ramsey1930problem}]
	For every $k, n \in \mathbb{N} = \{0, 1, \ldots \}$ there is some number $N \in \mathbb{N}$ such that if $X$ is a set and $E \subseteq X^k$ is a $k$-ary relation on $X$,  then every sequence of elements of $X$  of length $N$ contains an $E$-indiscernible subsequence of length $n$.
	
	We denote the smallest such $N$ by $R_k(n)$.
\end{fact}

Establishing exact bounds for the asymptotics of $R_k(n)$ is one of the central open problems in combinatorics, even in the case $k=2$. We summarize briefly some of the known results.

\begin{fact} \label{fac: general Ramsey bounds}
	\begin{enumerate}
		\item \cite{erdos1947some, erdos1935combinatorial} $2^{\frac{n}{2}} < R_2(n) < 2^{2n}$ for all $n>2$.
		\item \cite{erdHos1965partition, erdos1952combinatorial} There are positive constants $c$ and $c'$ such that $2^{cn^2} < R_3(n) < 2^{2^{c'n}}$ for all sufficiently large $n$.
		\item \cite{graham1990ramsey, conlon2010hypergraph} For each $k \geq 3$ there are positive constants $c,c'$ such that $\twr_{k-1}(cn^2) \leq R_k(n) \leq \twr_{k}(c'n)$ for all sufficiently large $n$, where the tower function $\twr_k(n)$ is defined recursively by $\twr_1(n) = n$ and $\twr_{i+1}(n) = 2^{\twr_i(n)}$.
	\end{enumerate}
\end{fact}

Recently, this question was investigated in the context of semialgebraic relations, where stronger bounds were obtained. Recall that a set $A \subseteq \mathbb{R}^d$ is \emph{semialgebraic} if it is given by a finite Boolean combination of sets of the form $\{ x \in \mathbb{R}^d : f(x) \geq 0\}$, where $f(x)$ is a polynomial in $d$ variables with coefficients in $\RR$. We say that a semialgebraic set $A$ has \emph{description complexity} at most $t$ if $d \leq t$ and $A$ can be written as  a Boolean combination of such sets involving at most $t$ different polynomials, each of degree at most $t$.

\begin{defn} \label{def: semialg R}

Let $E \subseteq (\mathbb{R}^d )^k$ be a $k$-ary semialgebraic relation on $\mathbb{R}^d$. For $n\in \NN$, we let $R_E(n)$ be the smallest natural number $N$ such that if $(a_i : 1\leq i \leq m), a_i \in \mathbb{R}^d$, is a sequence of length $m \geq N$, then it contains an \emph{$E$-indiscernible} subsequence of length $n$.

Let $R_{k}^{d,t}(n)$ be the maximum of $R_{E}(n)$, where $E$ varies over all $k$-ary semialgebraic relations on $\mathbb{R}^d$ of description complexity at most $t$.
\end{defn}

The case of binary relations ($k=2$) is addressed in the following theorem, which shows that $R_2^{d,t}(n)$ can be bounded by a polynomial in $n$ --- as opposed to the necessarily exponential bound in the general case (Fact \ref{fac: general Ramsey bounds}(1)). The following is proved in \cite[Theorem 1.2]{alon2005crossing} (it is only stated there for \emph{symmetric} semialgebraic relations; the result for arbitrary semialgebraic relations follows easily from the symmetric case using that the lexicographic ordering on $\mathbb{R}^d$ is semialgebraic --- see the discussion after Definition \ref{def: strongEH}).

\begin{fact} \cite[Theorem 1.2]{alon2005crossing} \label{fac: Alon et al}
	For any $d,t$ there is some $c = c(d,t)$ such that $R_2^{d,t}(n) \leq n^c$ for all sufficiently large $n$.
\end{fact}

Based on this fact, \cite{conlon2014ramsey} addresses the case of general $k$, establishing that $R_k^{d,t}(n)$ can be bounded from above by an exponential tower of height $k-1$ (as opposed to $k$ for general relations; Fact \ref{fac: general Ramsey bounds}(3)).

\begin{fact} \label{fac: semialgebraic tower}
 \cite{conlon2014ramsey} For any $k \geq 2$ and $d,t \geq 1$ there is some $c = c(k,d,t)$ such that $R_k^{d,t}(n) \leq \twr_{k-1}(n^c)$ for all sufficiently large $n$.
\end{fact}

In addition, matching lower bounds for semialgebraic relations were obtained in \cite{conlon2014ramsey} and subsequently refined in \cite{elias2014lower}.

\begin{fact} \label{fac: lower bounds semialg}
\begin{enumerate}

	\item \cite{conlon2014ramsey} For every $k\geq 4$, there exists $d = d(k)$, $t = t(k)$, $c' = c'(k)$ and a $k$-ary semialgebraic relation $E$ on $\RR^d$ of description complexity $\leq t$ such that $R_E(n) \geq \twr_{k-1}(c' n)$ for all sufficiently large $n$.
	\item \cite{elias2014lower} In (1), one can take $d = k-3$.
\end{enumerate}
	
\end{fact}

The dependence of the dimension $d$ on the arity $k$ of the relation $E$ in Fact \ref{fac: lower bounds semialg} is unavoidable, due to the following theorem of Bukh and Matousek.

\begin{fact} \cite{bukh2014erdHos} \label{fac: BukhMat}
	For every $k \in \NN$ and every $k$-ary semialgebraic relation $E$ on $\RR$ there is some $c = c(E)$ such that $R_E(n) \leq 2^{2^{cn}}$ for all sufficiently large $n$.
\end{fact}
That is, if we restrict to arbitrary $k$-ary semialgebraic relations on $\mathbb{R}$ (as opposed to $\mathbb{R}^d$ for some $d > 1$), then $R_E(n)$ is at most double exponential (rather than a tower of height $k-1$ as in Fact \ref{fac: semialgebraic tower}). The constant $c$ given by the proof in \cite{bukh2014erdHos} actually depends on the parameters of $E$ (and not just on its description complexity, as in Fact \ref{fac: semialgebraic tower}); however this dependence can be eliminated (see Theorem \ref{thm: BM for polybdd omin}).
\medskip

In this paper we investigate a generalization from semialgebraic relations to relations definable in more general first-order structures, and the connection between Ramsey growth for relations definable in a structure and the model-theoretic tameness conditions that this structure satisfies.

\begin{defn} \label{def: main}
Let $\mathcal{M}$ be a first-order structure in a language $\CL$ (we denote its underlying set by $M$). Let $k \geq 1$ be an integer and let $\varphi(x_1, \ldots, x_k)$ be an $\CL(M)$-formula (i.e. ~ a formula with parameters from $\CM$) with its free variables partitioned into $k$ groups of equal size, i.e. $|x_1| = \ldots = |x_k| = d$. Then $\varphi$ defines a $k$-ary relation $\varphi(M)$ on $M^d$ (a definable subset of $M^d$ in the case $k=1$), namely $\varphi(M) = \{ (a_1, \ldots, a_k) \in (M^d)^k : \mathcal{M} \models \varphi(a_1, \ldots, a_k) \}$. The case $|x_i|=1$ for all $i=1, \ldots, k$ will be referred to as ``$x_1, \ldots, x_k$ singletons''.

 We let $R_\varphi(n)$ be the smallest natural number $N$ such that any sequence $(a_i : 1\leq i \leq N), a_i \in M^d$, of length $N$ contains a \emph{$\varphi(M)$-indiscernible} subsequence of length $n$.
  
 Also, given an $\CL$-formula $\varphi(x_1, \ldots, x_k; z)$, where $|x_1| = \ldots = |x_k| = d$ and $z$ is an additional tuple of free variables, we let $R^*_{\varphi}(n) := \max \{ R_{\varphi(x_1, \ldots, x_k; b)}(n) : b \in M^{|z|} \}$ (or $\infty$ if the maximum doesn't exist).
\end{defn}

\begin{rem}
 By Tarski's quantifier elimination in the field of reals $\CM = (\mathbb{R}, <,  +, \times, 0,1)$, given a formula $\varphi(x; y)$, all sets of the form $\varphi(\mathbb{R}^{|x|}; b), b \in \mathbb{R}^{|y|}$, are semialgebraic of description complexity $\leq t$ for some $t$ depending only on $\varphi$.
 Conversely, the family of all semialgebraic subsets of $\mathbb{R}^{|x|}$ of description complexity $\leq t$ is of the form $ \{ \varphi(\mathbb{R}^{|x|}; b) : b \in \mathbb{R}^{|y|}\}$ for an appropriate choice of $\varphi(x;y)$. Hence $R^{d,t}_k$ from Definition \ref{def: semialg R} is given by  $R^*_{\varphi}$ for an appropriate $\varphi$ in the case of the field of reals.
\end{rem}

We will restrict to the case of \emph{NIP structures} (see Section \ref{sec: prelims on NIP} for the definition;  any structure which is not NIP codes arbitrary finite graphs in a definable way (see e.g. \cite[Remark 4.12]{chernikov2015regularity}), hence bounds in Fact \ref{fac: general Ramsey bounds} are optimal outside of the NIP context). First we give a brief overview of the relevant results in the model-theoretic literature indicating the relevance of NIP and its subclasses for the problem at hand.

The infinitary version of the problem of finding indiscernible subsequences was long known in model theory, under the name of the ``\emph{existence of indiscernibles}'' (starting with the work of Morley in the stable case, and later work of Shelah and others in general NIP \cite{ShelahCT, shelah1986around, shelah2014strongly, kaplan2014dependent}). 

The question of obtaining explicit bounds for $R_{\varphi}(n)$ under some model-theoretic tameness assumptions on $\mathcal{M}$ was first considered, it appears, in \cite{ensley1997ramsey}, where some quantitive improvements in the stable and NIP cases were obtained.
In the case of a stable formula $\varphi$, a polynomial upper bound was established in \cite{malliaris2014regularity}.

\begin{fact} \cite{malliaris2014regularity} \label{fac: stable EH}
Let $\varphi(x_1, \ldots, x_k; z)$ be a formula in a stable structure $\mathcal{M}$ (or just assume that $\varphi$ is a stable formula, relative to an arbitrary partition of its variables). Then there is some $c = c(\varphi)$ such that $R^{*}_{\varphi}(n) \leq n^c$ for all sufficiently large $n$.
\end{fact}

See also \cite{chernikov2015note} for a different proof using the ``non-standard'' method. Fact \ref{fac: Alon et al} was generalized to $o$-minimal structures (with some additional topological assumptions) in \cite{basu2010combinatorial}, and to \emph{symmetric} relations in arbitrary \emph{distal structures} in the following theorem (see Definition \ref{def: distality} for the definition of distality; examples of distal structures include arbitrary  $o$-minimal structures and $P$-minimal structures, e.g. the fields $\mathbb{Q}_p$ for $p$ prime --- see Definition \ref{def: p-minimal}).

\begin{fact} (\cite[Theorem 3.6]{chernikov2015regularity} + Remark \ref{rem: strong EH implies EH}) \label{fac: EH in distal}
Let $\mathcal{M}$ be a reduct of a distal structure. Then for any formula $\varphi(x_1, x_2; z)$ with $|x_1|=|x_2|$ arbitrary and such that the relation defined by $\varphi(x_1,x_2;b)$ is \emph{symmetric} for any $b \in M^{|z|}$, there is some $c = c(\varphi)$ such that $R^*_{\varphi} (n) \leq n^c$ for all sufficiently large $n$.
\end{fact}

In this paper, we continue investigating the bounds for the functions $R_{ \varphi}(n)$ and $R_{\varphi}^*(n)$ in various NIP structures. First, we consider an analog of the Bukh-Matousek theorem (Fact \ref{fac: BukhMat}) in $o$-minimal structures. Recall that a structure $\mathcal{M} = (M, <, \ldots)$ is \emph{$o$-minimal} if every definable subset of $M$ is a
finite union of singletons and intervals (with endpoints in $M\cup\{\pm\infty\}$). From this
assumption one obtains cell decomposition and other geometric information for definable subsets of $M^{n}$, for all $n$. The theory of $o$-minimal structures is rather well developed and has applications in other branches of mathematics (we refer to
\cite{van1998tame} for a detailed treatment of $o$-minimality, or to \cite[Section 3]{scanlonminimality}
and references therein for a quick introduction). Examples of $o$-minimal structures include
$\bar{\mathbb{R}} = \left(\mathbb{R},+,\times\right)$,
$\mathbb{R}_{\exp} = \left(\mathbb{R},+,\times,e^{x}\right)$,
$\mathbb{R}_{\text{an}} = \left(R,+,\times,f\restriction_{\left[0,1\right]^k}\right)$ for $f$
ranging over all functions that are real-analytic on some neighborhood of $[0,1]^k$, or $\mathbb{R}_{\text{an},\exp }$, the combination of both these last two examples. An $o$-minimal structure $\mathcal{M}$ is \emph{polynomially bounded} if for every definable one-variable function $f$, there exists $N \in \mathbb{N}$ such that $|f(x)| \leq x^N$ for all sufficiently large positive $x$. So for example $\bar{\mathbb{R}}$ and $\mathbb{R}_{\text{an}}$ are polynomially bounded, but $\mathbb{R}_{\exp}$ is not. In Section \ref{sec: BukhMat in polybdd} we generalize Fact \ref{fac: BukhMat} to arbitrary polynomially bounded $o$-minimal expansions of the field of reals $\bar{\RR}$.

\begin{thm} \label{thm: BM for polybdd omin}
	Let $\mathcal{M}$ be a polynomially bounded $o$-minimal expansion of  $\RR$. Then for every $k \in \NN$ and every formula $\varphi(x_1, \ldots, x_k;z)$ with $x_1, \ldots, x_k$ singletons, there is some $c = c(\varphi)$ such that $R^{*}_\varphi(n) \leq 2^{2^{cn}}$ for all sufficiently large $n$.
\end{thm}

In particular this implies that in the semialgebraic case (Fact \ref{fac: BukhMat}) the constant $c$ only depends on the description complexity of the relation, and not on the magnitude of the parameters, which doesn't seem to have been noticed before. Our argument combines uniform definability of types over finite sets in NIP structures (see Definition \ref{def: UDTFS}), basic properties of invariant types and a combinatorial lemma from \cite{bukh2014erdHos}. On the other hand, in Section \ref{sec: Counterex in Rexp} we show that no analog of Theorem \ref{thm: BM for polybdd omin} can hold in $\mathbb{R}_{\exp}$.

In this paper, ``$\log$'' always means logarithm with base $2$, unless explicitly stated otherwise.
\begin{thm}\label{thm: unbounded Rexp}
	For every $k \geq 3$ there are relations $E_k(x_1, \ldots, x_k)$ definable in $\mathbb{R}_{\exp}$ with $x_1, \ldots, x_k$ singletons, constants $C_k > 0$ and $n_k \in \mathbb{N}$ such that,  for each  $n > n_k$,  there is a sequence $\vec{a}_n$ in $\mathbb{R}$ of length $n$ that doesn't contain an $E_k$-indiscernible subsequence of length greater than $C_k \log \log \ldots \log n$, with $k-2$ iterations of $\log$.
\end{thm}

By a theorem of Miller \cite{miller1994exponentiation}, if an $o$-minimal expansion of the field of real numbers is not polynomially bounded, then exponentiation is definable in it (i.e. the graph of the exponentiation function is a definable relation). Combining this with Theorems \ref{thm: BM for polybdd omin} and \ref{thm: unbounded Rexp} we obtain a new combinatorial characterization of polynomial boundedness for $o$-minimal expansions of $\mathbb{R}$.

\begin{cor}
Let $\mathcal{M}$ be an $o$-minimal expansion of $\mathbb{R}$. The following are equivalent.
\begin{enumerate}
\item $\CM$ is polynomially bounded.
\item For every $k \in \NN$ and every formula $\varphi(x_1, \ldots, x_k;z)$ with $x_1, \ldots, x_k$ singletons, there is some $c = c(\varphi)$ such that $R^{*}_\varphi(n) \leq 2^{2^{cn}}$ for all sufficiently large $n$.
\item There is some $h \in \mathbb{N}$ such that, for every $k \in \NN$ and every formula $\varphi(x_1, \ldots, x_k;z)$ with $x_1, \ldots, x_k$ singletons,  there is some $c = c(\varphi)$ such that $R^{*}_\varphi(n) \leq \twr_{h}(n^c)$ for all sufficiently large $n$.

\end{enumerate}

\end{cor}

Using the general method of the proof developed in Section \ref{sec: BukhMat in polybdd}, in Section \ref{sec:case-p-adics} we apply it to prove an analog of Fact \ref{fac: BukhMat} in the fields of the $p$-adics $\mathbb{Q}_p$, for $p$ prime, and many related structures (see Section \ref{sec:case-p-adics} for the definition of $P$-minimality and related notions). 

\begin{thm}\label{thm: BM in P-min}
Let $\CM$ be a $P$-minimal expansion of a field with definable Skolem functions and the value group $\mathbb{Z}$. Then,  for every $k \in \NN$ and every formula $\varphi(x_1, \ldots, x_k;z)$, with $x_1, \ldots, x_k$ singletons, there is some $c = c(\varphi)$ such that $R^{*}_\varphi(n) \leq 2^{2^{cn}}$ for all sufficiently large $n$.

\end{thm}
This applies to the fields $\mathbb{Q}_p$ for all primes $p$, their finite extensions, as well as expansions by the analytic structure --- see Section \ref{sec:case-p-adics} for the details. In fact, there are no known examples of $P$-minimal structures with value group $\mathbb{Z}$ that do not satisfy Theorem \ref{thm: BM in P-min} (note that the combinatorial conclusion obviously transfers to the reducts).

\begin{prob}
Do Theorems \ref{thm: BM for polybdd omin} and \ref{thm: BM in P-min} hold in polynomially bounded $o$-minimal (respectively, $P$-minimal) theories that do not admit any archimedean models?
\end{prob}

In Section \ref{sec: Ramsey growth NIP} we consider the growth of $R^*_{\varphi}(n)$ in NIP structures for definable relations of higher arity. Generalizing Fact \ref{fac: semialgebraic tower}, we show a definable  stepping down lemma for NIP structures which implies the following.

\begin{thm}\label{thm: tower NIP}
Let $\mathcal{M}$ be an NIP structure, and assume that for all formulas $\varphi(x_1, x_2; z)$  we have $R^*_{\varphi}(n) \leq n^c$ for some $c = c(\varphi)$ and all $n$ large enough. Then for all $k \geq 3$ and all $\varphi(x_1, \ldots, x_k ; z)$ we have $R^*_{\varphi}(n) \leq \twr_{k-1}(n^c)$ for some $c = c(\varphi)$ and all $n$ large enough.
\end{thm}

In Proposition \ref{prop: sym EH implies EH} we generalize Fact \ref{fac: EH in distal} from symmetric binary formulas to arbitrary binary formulas, demonstrating that the assumption of Theorem \ref{thm: tower NIP} is satisfied in all reducts of distal structures (and it is satisfied in stable structures by Fact \ref{fac: stable EH}). We conjecture that it also holds in arbitrary NIP structures and discuss the connection to the Erd\H{o}s-Hajnal conjecture (see e.g.~ \cite{chudnovsky2014erdos}) for graphs definable in NIP structures.

\subsection*{Acknowledgements}
We would like to thank Martin Hils for encouraging us to present the results in Section \ref{sec:case-p-adics} in full generality, Pablo Cubides Kovacsics for a very helpful discussion  on polynomial boundedness in $P$-minimal structures, and the anonymous referee for some suggestions on improving the presentation.

Chernikov was supported by the NSF Research Grant DMS-1600796; by the NSF CAREER grant DMS-1651321 and by an Alfred P. Sloan Fellowship.

Starchenko was  supported by the NSF Research Grant DMS-1500671.

Thomas was supported by the DFG Research Grant TH 1781/2-1; by the Zukunftskolleg, University of Konstanz and by the NSERC Discovery Grant RGPIN 261961.

This
work was finished during the ``Model Theory, Combinatorics and Valued Fields'' trimester program
at Institut Henri Poincar\'e. We thank IHP for their
hospitality.

\section{Preliminaries on NIP} \label{sec: prelims on NIP}
Vapnik--Chervonenkis dimension, or VC-dimension, is an important notion in combinatorics and statistical learning theory (see e.g. \cite{matouvsek2002lectures} for an exposition).
Let $X$ be a set, finite or infinite, and let $\mathcal{F}$ be a
family of subsets of $X$.  Given $A \subseteq X$, we say that it is \emph{shattered} by
$\mathcal{F}$ if for every $A' \subseteq A$ there is some $S \in \mathcal{F}$ such that
$A \cap S = A'$. A family $\mathcal{F}$ is a \emph{VC-class} if there is some $n < \omega$
such that no subset of $X$ of size $n$ is shattered by $\mathcal{F}$. In this case \emph{the
  VC-dimension of $\mathcal{F}$}, that we will denote by $VC(\mathcal{F})$, is the smallest integer
$n$ such that no subset of $X$ of size $n+1$ is shattered by $\CF$.  For a set $B\subseteq X$, let
$\mathcal{F}\cap B=\left\{ A\cap B:A\in\mathcal{F}\right\}$ and let
$\pi_{\mathcal{F}}\left(n\right)=\max\left\{ \left|\mathcal{F}\cap B\right|:B\subseteq
  X,\left|B\right|=n\right\} $.

\begin{fact}[Sauer-Shelah lemma \cite{sauer1972density, shelah1972combinatorial}]\label{fac: SauerShelah} If $VC(\mathcal{F}) \leq d$ then for
  $n\geq d$ we have
  $\pi_{\mathcal{F}}\left(n\right)\leq\sum_{i\leq d}{n \choose i}=O\left(n^{d}\right)$.
\end{fact}

The important class of NIP theories was introduced by Shelah in his work on the classification
program \cite{ShelahCT}. It has attracted a lot of attention recently, both from the
point of view of pure model theory and as a result of its applications in algebra and geometry (see e.g. \cite{adler2008introduction, simon2015guide} for an introduction to the area). Examples of NIP structures are given by arbitrary stable structures, (weakly or quasi) $o$-minimal structures, the field of $p$-adics for every prime $p$ (along with its analytic expansion), as well as algebraically closed valued fields. As was observed in \cite{laskowski1992vapnik}, the original definition of NIP is equivalent to the
following one (see \cite{VCD1} for a more detailed account).
\begin{defn} Let $T$ be a complete theory and $\varphi(x,y)$ a formula in $T$, where $x,y$ are tuples
  of variables, possibly of different length. We say that \emph{the formula $\varphi(x,y)$ is NIP} if
  there is a model $\CM$ of $T$ such that the family of definable sets $\{ \varphi(M,a) : a \in M^{|y|} \}$ is a
  VC-class. In this case we define the \emph{VC-dimension of $\varphi(x,y)$} to be the VC-dimension of this
  class. (It is easy to see that by elementarily equivalence the above does not depend on the model
  $\CM$ of $T$.) \emph{A theory $T$ is NIP} if all formulas in $T$ are NIP, and a structure $\CM$ is NIP if its complete theory $\Th(\CM)$
is NIP. That is, a structure $\CM$ is \emph{NIP} if for every formula
$\varphi(x,y)$ the family of $\varphi$-definable sets
$\mathcal{F}_\varphi=\{\varphi(M,a) : a \in M^{|y|} \}$ is a VC-class.

\end{defn}
By a \emph{partitioned} set of formulas $\Delta(x,y)$, where $x$ and $y$ are two groups of variables, we mean a set of formulas all of which are of the form $\varphi(x,y) \in \CL$, i.e. have the same free variables partitioned into the same two groups. 
Given a (partitioned) set of formulas $\Delta(x,y)$ and a set $B \subseteq M^{|y|}$, we say that $\pi(x)$
is a \emph{$\Delta$-type over $B$} if
$\pi(x) \subseteq \bigcup_{\varphi(x,y) \in \Delta, b \in B} \left\{ \varphi(x,b), \neg \varphi(x,b) \right\}
$
and there is some $\mathcal{N} \succeq \CM$ and some $a \in N^{|x|}$ simultaneously satisfying all
formulas from $\pi(x)$. By a \emph{complete} $\Delta$-type over $B$ we mean a maximal $\Delta$-type
over $B$. We will denote by $S_{\Delta}(B)$ the collection of all complete $\Delta$-types over
$B$. If $\Delta$ consists of a single formula $\varphi(x,y)$, we simply say $\varphi$-type and write $S_\varphi(B)$, and if $\Delta$ consists of all formulas in the language, then we simply say ``type'' and write $S_x(B)$ for the space of complete types over $B$. In view of the remarks above, the following is an immediate corollary of the Sauer-Shelah
lemma.
\begin{fact}\label{fac: PolyTypesNIP} A structure $\CM$ is NIP if and only if for any finite set of
  formulas $\Delta(x,y)$ there is some $d \in \mathbb{N}$ such that $|S_{\Delta}(B)| = O(|B|^d)$ for
  any finite $B \subseteq M^{|y|}$.
\end{fact}

This result can be strengthened. The following definition is from \cite{guingona2012uniform, VCD1}.

\begin{defn} \label{def: UDTFS}

\begin{enumerate}
	\item Given a complete $\varphi(x,y)$-type
$q \in S_\varphi (B)$ for a set $B\subseteq M^{|y|}$, an $L(M)$-formula $d \varphi(y)$ is said to \emph{define} $q$ if for all $b \in B$ we have $$ \varphi(x,b) \in q \iff \CM \models d\varphi(b).$$
	\item We say that complete $\varphi(x,y)$-types are \emph{uniformly definable over finite sets}, with $m$ parameters, if there is a finite set of $\CL$-formulas $\Delta = (d\varphi_i(y;y_1, \ldots, y_m) : i < k)$, with $|y_1| = |y|$ for all $i<k$, such that for every \emph{finite} set $B \subseteq M^{|y|}$ and every $q \in S_\varphi(B)$ there are some $b_1, \ldots, b_m \in B$ and some $i<k$ such that $d\varphi_i (y; b_1, \ldots, b_m)$ defines $q$. We call the set $\Delta$ a \emph{uniform definition} for $\varphi$-types over finite sets, with $m$ parameters.
	\item We say that $T$ satisfies the \emph{Uniform Definability of Types over Finite Sets}, or \emph{UDTFS}, if for some (equivalently, any) $\CM \models T$, complete $\varphi$-types are uniformly definable over finite sets for all formulas $\varphi \in L$.
\end{enumerate}

\end{defn}

\begin{fact} \cite{chernikov2015externally} \label{fac: UDTFS}
	Every NIP theory satisfies UDTFS.
\end{fact}

This result can be viewed as a model-theoretic version of the Warmuth conjecture on the existence of compression schemes for VC-families, which was later established in \cite{moran2015sample}. Special cases of Fact \ref{fac: UDTFS} were proved earlier for some subclasses of NIP theories including stable \cite{ShelahCT}, $o$-minimal \cite{johnson2010compression}, and dp-minimal \cite{guingona2012uniform} theories. Note that this implies Fact \ref{fac: PolyTypesNIP} since, under UDTFS, for every finite set of formulas $\Delta$, every $\Delta$-type over a finite set $B$ is determined by fixing a definition for each $\varphi \in \Delta$ with parameters from $B$, of which there are only polynomially many choices. Explicit bounds on the number of parameters needed are given in \cite{VCD1} for some cases considered in this article.
\begin{fact} \label{fac: bounds on UDTFS}
\begin{enumerate}
	\item \cite[Section 6.1]{VCD1} Let $\CM$ be a (weakly or quasi) $o$-minimal structure. Then $\varphi(x,y)$-types are uniformly definable over finite sets using $|x|$ parameters, for all formulas $\varphi \in L$. In particular this applies to Presburger arithmetic $(\mathbb{Z},+,<)$.
	\item \cite[Section 7.2]{VCD1} Let $\CM$ be the field of $p$-adics. Then $\varphi(x,y)$-types are uniformly definable over finite sets using $2|x|$ parameters, for all formulas $\varphi \in L$.
\end{enumerate}
	
\end{fact}

Finally, we recall global invariant types and their products. We will use some standard model-theoretic notation, e.g. ~ $\mathbb{M} \succ \CM$ will be a saturated elementary extension, and, given a set $A \subseteq \mathbb{M}$, $\dcl(A)$ will denote the model-theoretic algebraic closure of $A$, $A$ will be called \emph{small} if its cardinality is smaller than the saturation of $\mathbb{M}$, etc. Given a tuple of variables $x$, we call complete types in $S_x(\mathbb{M})$ \emph{global}, and we say that a global type $p(x)$ is \emph{$M$-invariant} if it is $\Aut(\mathbb{M} / M)$-invariant (meaning that, for every automorphism $\sigma$ of $\mathbb{M}$ fixing $M$ pointwise, for every $L(\mathbb{M})$-formula $\varphi(x,a)$, we have $\varphi(x,a) \in p \iff \varphi(x, \sigma(a)) \in p$).

\begin{defn}
Given a set of formulas $\Delta$, $d \in \mathbb{N}$, a set of parameters $A \subseteq \mathbb{M}$ and an arbitrary linear order $I$, we say that a sequence 	$(a_i : i \in I)$ of tuples from $\mathbb{M}^d$ is \emph{$\Delta$-indiscernible over $A$} if it is $E$-indiscernible for every relation $E$ of the form $\varphi(x_1, \ldots, x_n;b)$ with $\varphi(x_1, \ldots, x_n;z) \in \Delta$, $|x_i| =d$ for all $1 \leq i \leq n$ and $b \in A^{|z|}$.

If $\Delta$ consists of all formulas, we simply say that the sequence is \emph{indiscernible over $A$}, and if $A = \emptyset$, we say that the sequence is $\Delta$-indiscernible.
\end{defn}

\begin{fact} (See e.g. \cite[Section 2]{hrushovski2011nip} or \cite{simon2015guide}) \label{fac: basic inv types}
	Let $p$ be a global $M$-invariant type. Let the sequence $(c_i : i \in \mathbb N)$ in $\mathbb{M}$ be such that $c_i \models p|_{Mc_{<i}}$ (such a sequence is called a \emph{Morley sequence} in $p$ over $M$). Then the sequence $(c_i : i \in \mathbb{N})$ is indiscernible over $M$ and $\tp((c_i : i \in \mathbb{N})/ M)$ does not depend on the choice of $(c_i)$. Call this type $p^{(\omega)}|_{M}$, and let $p^{(n)}|_M := \tp(c_1, \ldots, c_n / M)$.
\end{fact}

\section{Bukh-Matousek theorem in polynomially bounded $o$-minimal expansions of $\RR$} \label{sec: BukhMat in polybdd}

First we prove a general lemma about NIP structures, which is a finitary version of Shelah's ``shrinking of indiscernibles'' \cite{MR2062198}.

\begin{lem} \label{lem: finitary shrinking}
Let $\CM$ be an NIP structure, and let $\varphi(x_1, \ldots, x_{n};y)$ be a  formula with $|x_1| = \ldots = |x_{n}|=d$. Then there are some $k, l \in \mathbb{N}$ and a finite set of formulas $\Delta$ in the variables $x_1, \ldots, x_{l}$ with $|x_i|=d$ such that for any finite $\Delta$-indiscernible sequence $(a_i)_{i < N}$ in $M^d$ and any $b \in M^{|y|}$ there are $0=j_0<j_1<\dotsc <j_{k'}=N-1$ with $k'
\leq k$ such that for every $s\in\{0,\dotsc,k'-1\}$ the sequence $(a_i : j_s< i <j_{s+1})$ is 
 $\varphi(x_1,\ldots, x_{n},b)$-indiscernible. 

In particular, for any $N$ large enough and any $b\in M^{|y|}$, any finite $\Delta$-indiscernible sequence of elements in $M^d$ of length $N$ contains a $\varphi(x_1,\dotsc,x_n,b)$-indiscernible subsequence of length at least $\frac{N-(k+1)}{k}$.
\end{lem}
\begin{proof}
To simplify the notation we assume $d=1$. 

By  UDTFS (Fact \ref{fac: UDTFS}) applied to the formula $\varphi^{\op}(y; x_1,\dotsc,x_n) := \varphi(x_1,\dotsc,x_n;y)$,
there is 
a finite set of formulas $\Delta(x_1, \ldots, x_{n};\bar{x}_1, \ldots, \bar{x}_m)$ with $|{\bar{x}}_i|=n$ such that,  
for any finite set $A\subseteq M$ and $b\in M^{|y|}$, the $\varphi^{\op}$-type of $b$ over $A^n$  
is definable by an instance of some $\psi\in \Delta$ with parameters from $A^n$. That is, there are some $\bar c_1, \dotsc, \bar c_m\in A^n$,  such that, for all $a_1,\dotsc,a_n\in A$, we have 
$\models \varphi(a_1,\dotsc,a_n;b)$ if and only if  $\models\psi(a_1,\dotsc,a_n;\bar c_1, \dotsc, \bar c_m)$.

Writing each $n$-tuple $\bar x_i$, $i=1,\dotsc,m$, as $n$ single variables in every  $\psi\in \Delta$, we can view
$\Delta$ as a finite set of formulas in the variables $x_1,\dotsc,x_l$, where $l=n+mn$. 

Let $(a_i)_{i< N}$ be a finite $\Delta$-indiscernible sequence, $b\in M^{|y|}$, and $A=\{ a_i \colon i< N\}$.  
We choose $\psi\in \Delta$ and $c_{n+1},\dotsc c_l\in A$ such that for all $c_1,\dotsc c_n\in A$ we have 
$\CM \models\varphi(c_1,\dotsc,c_n;b)$ if and only if $\CM \models \psi(c_1,\dotsc,c_n,c_{n+1},\dotsc,c_l)$. 

We choose $0=j_0<j_1<\dotsc <j_{k'}=N-1$ with $k'\leq (l-n)+2=mn+2$ so that 
$\{a_{j_s} \colon  s=0,\dotsc,k'\}=\{ c_i \colon i=n+1,\dotsc,l\}\cup\{ a_0,a_{N-1}\}$.

Since $(a_i)_{i<N}$ is $\psi$-indiscernible, it follows that for any $0 \leq i_1 < \ldots < i_n < N$ the truth value of $\psi(a_{i_1}, \ldots, a_{i_{n}}; c_{n_1}, \ldots, c_l)$, and so of $\varphi(a_{i_1}, \ldots, a_{i_{n}}; b)$, is determined by the quantifier-free order type of $(i_1, \ldots, i_{n})$ over $\{ j_s : s=0,\dotsc k'\}$. 
The conclusion of the lemma follows taking $k := mn+2$.
\end{proof}

From now on we work in a polynomially bounded   $o$-minimal expansion  $\CR=\la \RR, <,  \dotsc\ra$ of the field of
real numbers. Let $T = \Th(\CR)$ and let $\mathbb{M} \succ \CR$ be a big saturated model.

As $T$ has Skolem functions (see e.g. \cite{van1998tame}), it follows that for all $\CM \prec \mathbb{M}$ and 
$\bar a\in\mathbb{M}^n$, the set 
$$ M \langle \bar a  \rangle = \{ f(\bar a) : f(x) \textrm{ is an } M \textrm{-definable function} \}$$
is an elementary substructure of $\mathbb{M}$.

Let $\tilde p(x) \in S_1(\mathbb{M})$ be the global type of ``$+ \infty$'', i.e. $\tilde p$ is the unique complete global type such that $\tilde p \vdash x>m$ for every $m\in \mathbb{M}$  (uniqueness is by $o$-minimality). It is  invariant over $\emptyset$ (as the set of formulas $\{ m < x : m \in \mathbb{M} \}$ is clearly $\Aut(\UU /\emptyset)$-invariant).

The following fact is obvious. 

\begin{fact}\label{fact:obvious}
For every $\CM \prec \mathbb{M}$, an element $\alpha\in \mathbb{M}$ realizes $\tilde p(x) |_M$ if and only if 
$\alpha >m$ for every $m\in M$.   
\end{fact}

Since  polynomial boundedness is preserved under elementarily equivalence (see \cite[Theorems A and B]{growth}) we have the following fact.
\begin{fact}  \label{fac: poly bdd growth}
If $\CM \prec \mathbb{M}$ and $\alpha \models \tilde{p}|_M$, then the set $\{ \alpha^n : n \in \mathbb{N} \}$ is cofinal in $M \langle \alpha \rangle$, i.e. for every $m \in M \langle \alpha \rangle $ there is some $n \in \NN$ such that $m < \alpha^n$.
\end{fact}

\begin{lem}\label{lem: infty1 iff R-growing}
  	Let $\CM \prec \mathbb{M}$ and $\alpha_1,\dotsc,\alpha_n \in \mathbb{M}$.  Then $(\alpha_1,\dotsc \alpha_n)$
realizes $\tilde p^{(n)}|_M$ if and only if $\alpha_1> m$ for all $m\in M$ and $\alpha_{i+1}>\alpha_i^k$ for all $k\in \NN$ and $i=1,\dotsc,n-1$. 
\end{lem}
\begin{proof}
Let $M_0=M$, and for $i=1,\dotsc n-1$ let $M_i=M_{i-1}\langle \alpha_i\rangle$.  

Obviously  for any $A \subset \mathbb M$ an element $\alpha\in \mathbb{M}$ realizes $\tilde p|_A$ if and only if it realizes 
$\tilde{p}|_{\dcl(A)}$. Thus  $(\alpha_1,\dotsc \alpha_n)$
realizes $\tilde p^{(n)}|_M$  if and only if $\alpha_{i+1}$ realizes $\tilde{p}|_{M_i}$ for $i=0,\dotsc,n-1$, and the lemma follows from Facts \ref{fact:obvious} and \ref{fac: poly bdd growth}.
\end{proof}

In view of the above  lemma, we define ``finitary'' approximations to a realization of $\tilde{p}^{(n)}|_\CR$.

\begin{defn}[Definition 2.1 \cite{bukh2014erdHos}] Let $h>2$ be a real number. A
  sequence $\vec a=(a_1,\dotsc,a_n)$  in $\CR$ is called {\em $h$-growing} if
  $a_1\geq h$ and $a_{i+1}\geq a_i^h$ for $i=1,\dotsc n-1$.
\end{defn}

Notice that any subsequence of an $h$-growing sequence is $h$-growing as well.

\begin{lem} \label{lem: o-min indisc without params}
For any finite set of formulas $\Delta(x_1,\dotsc,x_l)$ with
  parameters from $\RR$ there is some $ h \in
  \RR$ such that any $h$-growing sequence $(a_i : i=1, \ldots, N)$ of elements in
  $\RR$ is $\Delta$-indiscernible.
  \end{lem}
 
 \begin{proof}
   Consider the (partial) type 
 \[ \Sigma(x_1,\dotsc,x_{2l}) = \left\{ x_1 > n \land \bigwedge_{i=1}^{2l-1} (x_{i+1} > x_i^{n}) : n \in \NN \right\}.\]
  
 By Fact \ref{fac: basic inv types}, for any $N \in \NN$, if $(a_1, \ldots, a_N) \models \tilde{p}^{(N)}|_M$, then the sequence $(a_1, \ldots, a_N)$ is indiscernible.  Together with Lemma \ref{lem: infty1 iff R-growing} this implies that 
 \[ \Sigma(x_1, \ldots, x_{2l}) \vdash \psi(x_1, \ldots, x_l ) \leftrightarrow \psi(x_{i_1}, \ldots, x_{i_l} ) \]
    for any $1 \leq i_1 < i_2 < \ldots <i_l \leq 2l$ and $\psi \in \Delta$.
  By compactness, this holds with $\Sigma$ replaced by some finite subset $\Sigma_0$. But then, if $a_1, \ldots, a_N$ is an $h$-growing sequence and $h$ is larger than the largest $n$ appearing in $\Sigma_0$, then every increasing $2l$-tuple from $a_1, \ldots, a_N$ satisfies $\Sigma_0$, hence $a_1, \ldots, a_N$ is $\Delta(x_1, \ldots, x_l)$-indiscernible.
 \end{proof}
 
 Combining Lemma \ref{lem: o-min indisc without params} with Lemma \ref{lem: finitary shrinking} we can allow additional parameters in $\Delta$.

\begin{cor} \label{cor: o-min indisc with params}
For any finite set of formulas $\Delta(x_1,\dotsc,x_l; y)$ with
  parameters from $\RR$ there is some $ h \in
  \RR$ and $m \in \NN$ such that,  for any $h$-growing sequence of elements
  $\vec{a} =  (a_i  :  i=1, \ldots, N)$ in $\RR$ with $N$ large enough and for any $b \in \RR^{|y|}$, $\vec{a}$ contains a  $\Delta(x_1, \ldots, x_l; b)$-indiscernible subsequence of length $\frac{N}{m}$.
 \end{cor}

\begin{proof}
	For every $\varphi(x_1, \ldots, x_l; y) \in \Delta$, let $k_\varphi \in \NN$ and the finite set of formulas $\Delta_\varphi$ be as given by Lemma \ref{lem: finitary shrinking} for $\varphi$, and let $\Delta' = \bigcup_{\varphi \in \Delta} \Delta_\varphi$ and $k = \max \{ k_\varphi : \varphi \in \Delta \}$. Now by Lemma \ref{lem: o-min indisc without params} there is some $h$ such that every $h$-growing sequence $\vec{a} = (a_1, \ldots, a_N)$ of elements from $\RR$ is $\Delta'$-indiscernible. By Lemma \ref{lem: finitary shrinking}, for any $b \in \RR^{|y|}$ we can find an interval $[i^0, i^1]$ in  $[1, N]$ of length at least $\frac{N-(k|\Delta| -2)}{k|\Delta|}$ such that  the sequence $(a_i : i^0 \leq  i \leq i^1)$ is $\Delta(x_1, \ldots, x_l; c)$-indiscernible. 
We can take $m = 2k |{\Delta}|$.
\end{proof}

Finally, the following combinatorial lemma is from
\cite{bukh2014erdHos} (namely, Proposition 2.4 combined with Definition 2.3 there).

\begin{fact}
  \label{prop:comb} 
For every $n$ and $h \geq h_0$, where $h_0$ is a certain absolute constant, there exists $N \leq 2^{h^{2n}}$ such that
for any sequence $\vec a$ of length $N$ there is an $h$-growing
sequence $\vec b$ of length $n$ and $A,B\in \RR$ such that one of the
following sequences is a subsequence of $\vec a$.
\begin{enumerate}
\item $A+Bb_i$, $ i=1,\dotsc, n$.
\item $A+\dfrac{B}{b_i}$, $ i=1,\dotsc, n$.
\item $A+Bb_i$, $ i=n,\dotsc ,1$.
\item $A+\dfrac{B}{b_i}$, $ i=n,\dotsc, 1$.
\end{enumerate}
(Note: the order in (3) and (4) is reversed.)
\end{fact}

We are ready to prove the main result of the section, generalizing
\cite[Proposition 1.6]{bukh2014erdHos}.

\begin{thm}
  \label{thm:main}
Let $\CR$ be a polynomially bounded $o$-minimal expansion of the real field. Then for any formula
$\varphi(x_1,\dotsc,x_r; z)$ with parameters from $\RR$, with all $x_i$ singletons, there is a
constant $C=C (\varphi)$ such that 
\[R^*_\varphi(n) \leq 2^{2^{Cn}}, \]
for all sufficiently large $n$. 
\end{thm}
\begin{proof}
Let $\Delta(x_1, \ldots, x_r; y_1, y_2,z)$ consist of the formulas
\[ \varphi_1(x_1,\dotsc,x_r;
y_1,y_2,z)=\varphi(y_1+y_2x_1,\dotsc,y_1+y_2x_r;z), \]
\[ \varphi_2(x_1,\dotsc,x_r;
y_1,y_2,z)=\varphi(y_1+\frac{y_2}{x_1},\dotsc,y_1+\frac{y_2}{x_r};z), \]
\[ \varphi_3(x_1,\dotsc,x_r;
y_1,y_2,z)=\varphi(y_1+y_2x_r,\dotsc,y_1+y_2x_1;z), \]
\[ \varphi_4(x_1,\dotsc,x_r;
y_1,y_2,z)=\varphi(y_1+\frac{y_2}{x_r},\dotsc,y_1+\frac{y_2}{x_1};z), \]
and let $h$ and $m$ be as given by Corollary \ref{cor: o-min indisc with params} for $\Delta$. Now assume that $\vec{a}$ is an arbitrary sequence of singletons of length $N = 2^{h^{2mn}}$ (which is bounded by $2^{2^{Cn}}$ for an appropriate constant $C$ depending just on $m,h$), and let $d\in \RR^{|z|}$ be an arbitrary  tuple of elements.

By Fact \ref{prop:comb}, there is some $h$-growing sequence $\vec{b}=(b_i : 1\leq i \leq mn)$ and some $A,B \in \RR$ such that one of the corresponding sequences given by (1)--(4) in Fact \ref{prop:comb} is a subsequence of $\vec{a}$.
By Corollary \ref{cor: o-min indisc with params}, $\vec{b}$ contains a $\Delta(x_1, \ldots, x_r; A,B,d)$-indiscernible subsequence of length $n$. But by the choice of $\Delta$, the corresponding subsequence of $\vec{a}$ must be $\varphi(x_1, \ldots, x_r;d)$-indiscernible.
\end{proof}

\section{Counterexample in $\mathbb{R}_{\exp}$} \label{sec: Counterex in Rexp}

\subsection{Preliminaries}
\label{sec:preliminaries}

We work in the structure $\CM := \Rexp$ in the language $\CL := (<, +,\times, 0,1, \exp(x))$, i.e the  expansion of the field of reals with the exponential
function. It is well known to be $o$-minimal \cite{wilkie1996model}.

Instead of tower notations we use iterated $\log$ and $\exp$.  By
induction on $n$ we define functions $e_n(x)$ and $l_n(x)$ as 
\[ e_0(x)=x,  e_{n+1}(x)=2^{e_n(x)}; \text{ and } l_0(x)=x, 
l_{n+1}(x)=\log (l_n(x)),
 \]
where by $\log$ we always mean $\log_2$.
Obviously $l_n(x)$ is defined for large enough $x$ and it is the compositional inverse  of $e_n(x)$.

Our goal is to prove the following theorem.

\begin{thm}
  \label{thm:mainexp} 
For every $k\geq 3$ there is a relation $E_k(x_1,\dotsc x_k)$ definable in $\Rexp$, with
$x_1, \ldots, x_k$ all singletons, and $c_k>0$ such that 
$R_{E_k}(n)\geq e_{k-2}(c_k
n)$ for all sufficiently large $n$. 
\end{thm}

The proof of the above theorem closely follows the proof of Theorem
1.2 in  \cite{conlon2014ramsey} (see also Theorem 1.3 in \cite{elias2014lower}). In general, the so-called stepping-up lemma of Erd\H{o}s and Hajnal \cite{graham1990ramsey, conlon2010hypergraph} gives a lower bound for $(k+1)$-ary relations which is exponentially larger than the one for $k$-ary relations. In \cite{conlon2014ramsey} it is demonstrated that the stepping-up lemma can be carried out ``semialgebraically'' (the $k$-ary semialgebraic relations that they construct live on $\mathbb{R}^d$, and $d$ grows with $k$, see Fact \ref{fac: lower bounds semialg}). We show that in the structure $\Rexp$ the stepping-up approach can be implemented definably without increasing the dimension (i.e.~ our $k$-ary relations all live on $\mathbb{R}$). But first we discuss some preliminaries.

\subsection{Robustness}
\label{sec:robustness-1}

We will use the notion of robustness from \cite{elias2014lower} (that was
originally called ``depth'' in \cite{conlon2014ramsey}).

\begin{defn}
Let $\varphi(x_1,\dotsc,x_k)$ be an $\CL$-formula and let $\vec a=(a_1,\dotsc, a_n)$ be a
sequence of real numbers. We say that $\varphi$ is \emph{robust} on $\vec a$ if there is
$\varepsilon>0$ such that, for all $1\leq i_1<\dotsb < i_k\leq n$ and all
real numbers $a'_1,\dotsc, a'_k$ with $|a_{i_j} -a'_j|< \varepsilon$ for each $j=1,\ldots,k$, we have 
\[ \models \varphi(a'_1,\dotsc,a'_k)\leftrightarrow
\varphi(a_{i_1},\dotsc,a_{i_k}). \]  
\end{defn}

\subsection{$\log_T$-transformations}
\label{sec:delta-log_t-transf}

\begin{defn} Let $\varphi(x_1,\dotsc,x_r)$ be an $\CL$-formula.   Let $T>0$ be a real
 number.  For a formula
 $\psi(y_1,\dotsc, y_s)$ we say that
 $\psi$ is a
 \emph{$\log_T$-transformation} of $\varphi$ if it is obtained from
 $\varphi$ by replacing
 \textbf{every} free variable $x_{i}$ in $\varphi$ by an expression of the form
 $\log_T(u_i-v_i)$ with $u_i,v_i\in \{ y_1,\dotsc,y_s\}$.

\end{defn}

\begin{defn}
We say that an $\CL$-formula $\varphi(x_1, \ldots, x_r)$ is an
\emph{$\rd$-formula} if it 
 depends only on the ratios of differences of its
variables, i.e. 
it is equivalent to a
  formula of the form $$\psi \left(\dfrac{x_{i_1}-x_{j_1}}{x_{p_1}-x_{q_1}},\dotsc, \dfrac{x_{i_s}-x_{j_s}}{x_{p_s}-x_{q_s}} \right) $$ for some $\psi(y_1, \ldots, y_s) \in \CL$, where  $i_t, j_t, p_t, q_t \in \{ 1, \ldots, r\}$ for all $t = 1, \ldots, s$ (and there are no other  free variables in $\psi$).
\end{defn}

\begin{claim}\label{claim:nice} Let $T>0$. 
A  $\log_T$-transformation of an $\rd$-formula  $\varphi(x_1, \ldots, x_r)$, is
   also an $\rd$-formula, and it is also a $\log_2$-transformation of
   $\varphi$.
\end{claim}
  \begin{proof}
In a  $\log_T$-transformation  of $\varphi$ an expression of the form 
 $\frac{x_i-x_j}{x_p-x_q}$  is replaced by an expression of the form 
\[ \frac{\log_T(u_i - v_i) -\log_T(u_j-v_j)}{\log_T(u_p -
  v_p) -\log_T(u_q-v_q)}, \]
which is equivalent to 
\[
\frac{\log_T\frac{u_i-v_i }{u_j-v_j}}
{\log_T\frac{u_p-v_p}{u_q-v_q}}.\]
Since the ratio of two logarithms does not depend on the base, it is
also equivalent to 
\[
\frac{\log\frac{u_i-v_i }{u_j-v_j}}
{\log\frac{u_p-v_p}{u_q-v_q}}.\]
 
Thus, a $\log_T$-transformation of an rd-formula $\varphi$ is again an rd-formula that is also a $\log_2$-transformation of $\varphi$. 
\end{proof}

\subsection{Proof of Theorem \ref{thm:mainexp}}
\label{sec:proof-theorem}

For a formula $\varphi(x_1,\dotsc,x_k) \in \CL$ with $|x_1|=\dotsb=|x_k|=d$
and an integer $n$ we will denote by $R_\varphi^+(n)$ the smallest
integer $N$ such that any \emph{increasing} sequence $a_1<\dotsb <a_N$
contains a $\varphi$-indiscernible subsequence of length $n$. 

Obviously for any formula $\varphi(x_1,\dotsc,x_k)$ with $|x_i|=1$ we
have  $R^+_\varphi(n)\leq R_\varphi(n)$. 

Thus Theorem \ref{thm:mainexp} follows from the following  refined
version.

\begin{thm}
  \label{thm:main1}
For every $k\geq 3$ there are an $\rd$-formula $E_k(x_1,\dotsc x_k) \in \CL$ with
$x_1, \ldots, x_k$ all singletons and a constant $C_k>1$  such that, for all real 
$0<c<1$ and for all  large enough 
$n \in \NN$, there is an increasing sequence of natural numbers ${\vec a}^{\,n}$ of
length at least $e_{k-2}(cn)$  such that $E_k$ is robust on ${\vec a}^{\,n}$, and ${\vec a}^{\,n}$ does not contain an $E_k$-indiscernible subsequence of
length $C_kn$. 
\end{thm}
\begin{proof}[Proof of Theorem \ref{thm:main1} $\Rightarrow$ Theorem \ref{thm:mainexp}]
Fix $0<c<1$ and set $c_{k} = \frac{c}{C_{k}}$, for each $k \geq 3$, where $C_{k}$ is the constant given by \ref{thm:main1}. We then have, for the rd-formula $E_k(x_1,\dotsc x_k)$ given to us by \ref{thm:main1} and for all  large enough 
$n$,  an increasing sequence of natural numbers ${\vec a}^{\,n}$ of
length at least $e_{k-2}(c_{k}n)$ such that ${\vec a}^{\,n}$ does not contain an $E_k$-indiscernible subsequence of
length $n$. Thus $R^+_\varphi(n) \geq e_{k-2}(c_{k}n)$, and hence $R_\varphi(n)  \geq e_{k-2}(c_{k}n)$ by the preceding remark.
\end{proof}

\begin{rem} To prove Theorem \ref{thm:main1} it is enough to construct
  formulas $E_k(x_1,\dotsc,x_k)$ whose truth values  are well defined only on increasing sequences 
    of real numbers $r_1<\dotsb<r_k$.  (The formula
  $\log(x_2-x_1)>\log(x_3-x_2)$ is an example of a formula that we
  will use often.)
\end{rem}

We proceed by induction on $k$.

\subsection{The base case $k=3$}
\label{sec:base-case-k=3}

For the following claim see \cite[Section 3.1]{conlon2014ramsey}.
\begin{claim}\label{claim:k=3}
Let $E_3(x_1,x_2,x_3)$  be the formula  $x_1+x_3-2x_2 \geq 0$.  Then
for any $n\geq 1$ the sequence $1,2,3,\dotsc,2^n$ does not contain an
$E_3$-indiscernible subsequence of length $n+2$. 
\end{claim}

It is not hard to see that $E_3$ is equivalent to an rd-formula.
Indeed we can rewrite $E_3$ as $x_3-x_2 \geq x_2-x_1$, which on increasing
sequences is equivalent to $\frac{x_3-x_2}{x_2-x_1} \geq 1$. 

We also need $E_3$ to be robust on ${\vec a}^{\,n}$. 
It is not hard to see that $E_3$ is not robust on the sequence
$1,2,\dotsc,2^n,$ since $1+3-2\cdot2=0$ and the truth of $E_3$ can change even if
we perturb the first 3 elements of the sequence by arbitrarily small positive amounts.  It is however also
easy to see that $E_3$ is robust on any sequence that does not
contain any terms
$a<b<c$ with $a+c-2b=0$, i.e. it is robust on any sequence that does
not contain a non-trivial  3-term arithmetic progression.  To get
such a sequence we use 
Behrend's Theorem (see \cite{Behrend}).

\begin{thm}[Behrend's Theorem]
  \label{thm:behrend} There is a constant $D>0$ such that 
for all natural numbers $m$ there exists a set $X\subseteq \{1,\dotsc,m\}$ with
$|X|\geq \dfrac{m}{2^{D\sqrt{\log m}}}$ not containing any 
non-trivial $3$-term arithmetic progressions. 
\end{thm}

For any $0<c<1$ and for all $n$ large enough, $2^{n-D\sqrt{\log{2^n} }} > 2^{cn}$. 
Therefore, for all large enough $n$, the sequence 
$ 1,2,\dotsc, 2^n$ contains a subsequence of length  $2^{cn}$ that does not
contain a non-trivial $3$-term arithmetic progression.  

\medskip
This finishes the case $k=3$, and we can take  $C_{3} := 2^\eta$, for any $\eta>0$    (as then  $C_{3}n
\geq n+2$ for all large enough  $n$).
\subsection{Inductive Step}
\label{sec:inductive-step}

Assume we have an $\rd$-formula $E_k(x_1,\dotsc,x_k)$ as in Theorem
\ref{thm:main1}. To complete the inductive step it is enough to construct an $\rd$-formula $E_{k+1}(x_1,\dotsc,x_{k+1} )$ satisfying the following for any $N \in \mathbb{N}$:\\

Let $\vec a$ be  an increasing  sequence of natural numbers of length
$N$  such that $E_k$ is robust on $\vec a$, and  
$\vec a$ does not contain an $E_k$-indiscernible subsequence of length
$n$.  Then there is an increasing  sequence of natural numbers  $\vec b$ of length  $2^N$
such that  $E_{k+1}$ is robust on $\vec b$ and  
$\vec b$ does not contain an $E_{k+1}$-indiscernible subsequence of length
$2n+k-4$. 

(We are then done taking $C_{k} := 2^{{k-3}+\eta}$ for all $k \geq 3$, where $\eta >0 $ was fixed in the base case.)

\medskip

Let  $\vec a=(a_1,\dotsc,a_{N})$ be an
increasing sequence of natural numbers such that $E_k$ is robust on
$\vec a$, and $\vec a$ does not contain an $E_k$-indiscernible
sequence of length $n$. 

Let $T$ be a very large integer,
specified later (in terms of $\vec{a}$).  

Consider the set 
\[  B_T=\left\{ \sum_{i=1}^N \beta_i T^{a_i} \colon \beta_i\in
  \{0,1\}\right\}. \]
Since $T$ is large enough, any $b\in B_T$ can be written uniquely as 
$b=\sum_{i=1}^N b(i) T^{a_i}$ with $b(i)\in \{0,1\}$. 
 Obviously $B_T$ has size $2^N$ and we construct the sequence $\vec b_T$ by
taking the increasing enumeration of $B_T$. 

For $b,c \in B_T$ with $b \neq c$, let $\Delta(b,c) :=\max\{ i \colon b(i)\neq c(i)\}$.
It is easy to see that, when $T$ is large enough, for $b, c \in B_T$ with $b \neq c$ 
and $i : =\Delta(b,c)$
we have $b< c
\Leftrightarrow b(i) < c(i)$. It follows then that
\begin{equation}
  \label{eq:1}
  b<c<d\in B_T \Rightarrow \Delta(b,c) \neq \Delta(c,d).
\end{equation}

Finally for $b\neq c \in B_T$ let $\delta(b,c) : =a_{\Delta(b,c)}$.

We will now construct the step-up relation
$E^\uparrow_{k}(x_1,\dotsc,x_{k+1})$ (not definable in $\Rexp$) on
increasing $(k+1)$-tuples of elements of $B_T$ (we don't care how it is defined on the other
elements).

Let $b_1< b_2<\dotsc < b_{k+1}$ be elements of $B_T$ and for
$i=1,\dotsc,k$ let $\delta_i :=\delta(b_{i+1},b_i)$.  Notice that
$\delta_i$ is an element of $\vec a$. 

We define $E^\uparrow_{k}(b_1,\dotsc,b_{k+1})$ to be true if and only if 
\[
\begin{array}{ccc} 
   E_k(\delta_1,\dotsc,\delta_k) & \text{ and } & 
                                                \delta_1<\delta_2<\dotsc <
                                                \delta_k, \\
& \text{ or } \\
 E_k(\delta_k,\dotsc,\delta_1) & \text{ and } &  \delta_1> \delta_2 >\dotsc >
                                                \delta_k, \\
& \text{ or } \\ 
\delta_1 < \delta_2     &\text{ and } &\delta_2> \delta_3.
\end{array}
\]

\begin{claim}
  \label{cliam:step-up-arg}
The sequence $\vec b_T$ does
  not contain an $E^\uparrow_{k}$-indiscernible subsequence of length $2n+k-4$.  
\end{claim}
\begin{proof} We repeat the Erd\H{o}s-Hajnal argument (see \cite[Lemma 3.1]{conlon2014ramsey}).

Assume, towards getting a contradiction, that   $\vec b_T$  contains
an $E^\uparrow_{k}$-indiscernible subsequence $c_1<c_2<\dotsb <c_{2n+k-4}$.
Let $\delta'_i :=\delta(c_{i+1},c_i)$.

Assume  first that there exists  $j$ such that $\delta'_j, \delta'_{j+1},\dotsc,\delta'_{j+n-1}$ is  a monotone sequence.
Then by \eqref{eq:1} this sequence must be strictly monotone.  From the definition of $E^\uparrow_{k}$ it follows then that the sequence 
$\delta'_j, \delta'_{j+1},\dotsc,\delta'_{j+n-1}$ is $E_k$-indiscernible --- a contradiction.

Thus neither of the sequences   $\delta'_1,\dotsc, \delta'_{n}$
or $\delta'_{n-1},\dotsc, \delta'_{2n-2}$ is monotone. Hence each of them  contains either a local maximum, i.e. $\delta'_{j-1} < \delta'_j > \delta'_{j+1}$,  or a local minimum, i.e. $\delta'_{j-1} > \delta'_j < \delta'_{j+1}$.
Since between two local minima there is a local maximum and vice versa, the sequence  $\delta'_1,\dotsc, \delta'_{2n-2}$ contains both a local maximum and a local minimum. But then, by the definition of $E^\uparrow_{k}$, the sequence 
$c_1<c_2<\dotsb <c_{2n+k-4}$ cannot be $E^\uparrow_{k}$-indiscernible.
A contradiction.

\end{proof}

\subsubsection{Definability}
\label{sec:definability}

Now, as in \cite{conlon2014ramsey}, for $b>c$ we define 
\[\bar \delta_T(b,c)=\log_T(b-c). \]   

It is not hard to see that for any fixed $\varepsilon>0$, if $T$ is
large enough, then for all $b>c\in B_T$ we have 
  $|\delta(b,c)-\log_T(b-c)|< \varepsilon$.   

Since $E_k$ is robust on $\vec a$, choosing a very large integer $T$
and considering the relation $E^{\uparrow T}_{k}(x_1,\dotsc,x_{k+1})$ obtained from
$E^\uparrow_{k}$ by replacing  $\delta_i$ by $\bar\delta_T(b_{i+1},b_i)$ for all $i$, we
obtain that for $b_1,\dotsc,b_{k+1}\in B_T$ with $b_1<\dotsc<b_{k+1}$ we have
$E^\uparrow_{k}(b_1,\dotsc,b_{k+1})$ if and only if
$E^{\uparrow T}_{k}(b_1,\dotsc,b_{k+1})$.
Hence $\vec b_T$ does not contain an $E^{\uparrow T}_{k}$-indiscernible subsequence
of length $2n+k-4$. 

Notice that $E_{k}^{\uparrow T}$ is definable in  $\Rexp$ and 
for $b_1<b_2<\dotsc<
b_{k+1}$ we have 
that $E^{\uparrow T}_{k}(b_1,\dotsc,b_{k+1})$ holds  if and only if 
\[
\begin{array}{ccc} 
   E_k(\bar\delta_T(b_2,b_1),\dotsc,\bar\delta_T(b_{k+1},b_k)) & \text{ and } & 
                                              \bigwedge_{i=1}^{k-1}
                                                                                \bar\delta_T(b_{i+1},b_i)<\bar\delta_T(b_{i+2},b_{i+1})  \\
                                                                               
& \text{ or } \\
    E_k(\bar\delta_T(b_{k+1},b_k),\dotsc,\bar\delta_T(b_{2},b_1)) &
                                                                    \text{ and } &  \bigwedge_{i=1}^{k-1}
                                                                                \bar\delta_T(b_{i+1},b_i)>\bar\delta_T(b_{i+2},b_{i+1}) \\

 & \text{ or } \\
\bar\delta_T(b_2,b_1) < \bar\delta_T(b_3,b_2)  & \text{ and } & \bar\delta_T(b_3,b_2)  > \bar\delta_T(b_4,b_3). 
\end{array}
\]

\begin{claim}
  \label{claim:rdE}
$E^{\uparrow T}_k$ is equivalent to an $\rd$-formula and does not
depend on $T$.
\end{claim}
\begin{proof}
 By definition, $E_{k}^{\uparrow T}$ is a Boolean combination of
 $\log_T$-transformations of $E_k$  and  formulas of the form $\log_T(y-x)>\log_T(u-v)$. 

By Claim \ref{claim:nice}, a $\log_T$-transformation of an $\rd$-formula
is an $\rd$-formula that does not depend on $T$, and so we only need to check
that  $\log_T(y-x)>\log_T(u-v)$ is equivalent to an $\rd$-formula that
does not depend on $T$. Indeed, 
 $\log_T(y-x)-\log_T(u-v)>0$ is equivalent to
 $\frac{y-x}{u-v}>1$, which is an $\rd$-formula. 
\end{proof}

Using Claim \ref{claim:rdE}, we define $E_{k+1}$ to be $E^{\uparrow
  2}_k$. We can write a more explicit definition of $E_{k+1}$. It is the
disjunction of three formulas $\varphi_1\vee \varphi_2\vee \varphi_3$, where
\[ \varphi_1  \text{ is } 
   E_k \Bigl(\log(x_2-x_1),\dotsc,\log(x_{k+1}-x_k)\Bigr) \wedge
   \bigwedge_{i=1}^{k-1} \Biggl(\frac{x_{i+1}-x_i}{x_{i+2}-x_{i+1}}<1
   \Biggr), \]
\[ \varphi_2  \text{ is } 
   E_k \Bigl(\log(x_{k+1}-x_k),\dotsc,\log(x_{2}-x_1)\Bigr) \wedge
   \bigwedge_{i=1}^{k-1} \Biggl(\frac{x_{i+1}-x_i}{x_{i+2}-x_{i+1}}>1
   \Biggr), \]
and 
\[ \varphi_3 \text{ is } \frac{x_2-x_1}{x_3-x_2}<1 \,\wedge\,
\frac{x_3-x_2}{x_4-x_3}>1. \]

\medskip
It remains to show that for large enough $T$ the relation $E_{k+1}$ 
is robust on $\vec b_T$.

\subsubsection{Robustness}
\label{sec:robustness}

It is not hard to see that since $E_k$ is robust on $\vec a$ and
$\log_T$ is continuous, both    $E_k(\log_T(x_2-x_1),\dotsc,\log_T(x_{k+1}-x_k))$
and $E_k(\log_T(x_{k+1}-x_k),\dotsc,\log_T(x_2-x_1))$ are robust on
$\vec b_T$, and we only need to check that all  of the formulas $x_{i+1}-x_{i} <
x_{i+2}-x_{i+1}$ and $x_{i+1}-x_{i} >x_{i+2}-x_{i+1}$ are robust on $\vec b_T$, i.e. for $b<c<d$ in $B_T$ we
don't have $c-b=d-c$.
It is easy to check that there are no such $b,c,d$ in $B_T$.

\section{Bukh-Matousek in expansions of the $p$-adics}
\label{sec:case-p-adics}

In this section we give an analog of Theorem 
  \ref{thm:main} for relations definable in the fields of the $p$-adic numbers $\mathbb{Q}_p$ for $p$ prime and many of their expansions. We begin by recalling the relevant definitions and facts.

 Let $\CL_p$ be the Macintyre language for the p-adics \cite{Mac}, i.e. $\CL_p$
consists of 
\begin{enumerate}[(a)]
\item the language of rings: $(0,1,+,-,\cdot, \ ^{-1})$;
\item a unary predicate $V$;
\item a unary predicate $P_n$ for each $n\in \NN$;
\end{enumerate}
with the usual interpretations in $\QQ_p$:  $V(\QQ_p)=\ZZ_p$ and
$P_n(\QQ_p) =\{ x\in \QQ_p\colon  \exists y \,x=y^n\}$. 
We will denote by $T_p$ the complete theory $\Th(\QQ_p)$.  Given $a \in \mathbb{Q}_p$, we will write $v(a)$ to denote the $p$-adic valuation of $a$; note that the relation $v(x) < (y)$ is definable in $\CL_p$.

By a result
of Macintyre (see \cite{Mac}), the theory $T_p$ eliminates quantifiers in the language $\CL_p$.
Similarly to the $o$-minimal case, there is a notion of minimality for expansions of $p$-valued fields. Recall that a \emph{$p$-valued field} $\mathcal{K}$ is a valued field of characteristic $0$ with the residue field of characteristic $p$, and such that $\mathcal{O}/p\mathcal{O}$ has finite dimension as a vector space over $\mathbb{F}_p$, where $\mathcal{O}$ is the valuation ring of $\CK$.

\begin{defn} \label{def: p-minimal}\cite{haskell1997version}
	Let $\CK$ be a $p$-valued field, viewed as a structure in the language $\CL_p$. An expansion $\CM$ of $\CK$  in a language $\CL \supseteq \CL_p$ is 
\emph{$P$-minimal} if, in every
model of $\Th(\CM)$, every definable subset in one variable is quantifier-free definable just using
the language $\CL_p$. 
\end{defn}
\begin{sample}\label{ex: P-minimal structs}
Important examples of $P$-minimal structures are given by:
\begin{enumerate}
\item for each prime $p$, the field $\mathbb{Q}_p$ (by Macintyre's theorem);
\item any finite extension of $\mathbb{Q}_p$ \cite{prestel1984formally};
\item given a finite extension of $\mathbb{Q}_p$, the expansion obtained by adding a new function symbol for every restricted analytic function \cite{van1999one}.
\end{enumerate}
\end{sample}

\begin{fact}\cite{haskell1997version}\label{fac: P-minimal fields}
Every $P$-minimal field $\CK$ is \emph{$p$-adically closed}, i.e. it is henselian and its value group is elementarily equivalent to $\mathbb{Z}$ as an ordered group.

In particular, if the value group is $\mathbb{Z}$, then $\CK$ is a finite extension of $\mathbb{Q}_p$, hence the residue field is finite.
\end{fact}

In this section we will prove the following.
\begin{thm} \label{thm: BM in P-min gen}
Let $\CM$ be a $P$-minimal expansion of a field, and assume that $\CM$ has definable Skolem functions and the value group of $\CM$ (i.e. the value group of the underlying $p$-adically closed field) is $\mathbb{Z}$. Then for any formula
$\varphi(x_1,\dotsc,x_r; z) \in \CL(\CM)$, with all $x_i$ singletons, there is a
constant $C=C (\varphi)$ such that 
\[R^*_\varphi(n) \leq 2^{2^{Cn}} \]
for all sufficiently large $n \in \NN$.

\end{thm}

It is well known that all of the structures in Example \ref{ex: P-minimal structs} satisfy the assumptions of the theorem.

~

Given a $P$-minimal expansion of a field $\CM$, we will write $\Gamma_{\CM}$ to denote its value group. It is well-known that $\Gamma_{\CM}$ is interpretable in $\CM$.
\begin{defn}(\cite[Definition 4.4]{kovacsics2018definable})\label{def: poly bdd P-min}
A $P$-minimal structure $\CM$ is \emph{uniformly polynomially bounded} if, for all definable sets $X, W$ and every definable family of functions $f:X \times W \to M$, there is some $n \in \mathbb{N}$ and a definable function $a: W \to \Gamma_{\CM}$ such that for each $w \in W$ we have $v(f_w(x)) > n v(x)$ for all $x \in X$ with $v(x) < a(w)$. 
\end{defn}

The next fact is immediate from \cite[Lemma 4.3]{darniere2017cell} (as their  ``Extreme Value Property'' holds in every $P$-minimal expansion of a field elementarily equivalent to one with the value group $\mathbb{Z}$; see the discussion in \cite[Page 123]{darniere2017cell})  and compactness. 
\begin{fact}\label{fac: P-min poly bdd}
Let $\CM$ be a $P$-minimal expansion of a field elementarily equivalent to a structure with the value group $\mathbb{Z}$. Then $\CM$ is uniformly polynomially bounded.
\end{fact}

From now on, we fix  a $P$-minimal expansion $\CM_0$ of a field  with the value group $\Gamma_{\CM_0} = \mathbb{Z}$ and definable Skolem functions in a language $\CL$. We will denote by $T$ the complete theory of $\CM_0$, and  we also fix  a large sufficiently saturated
and homogeneous model $\UU$ of $T$. We are following the same strategy as in Section \ref{sec: BukhMat in polybdd}. First we isolate some sufficiently representative global invariant types (in the $o$-minimal case, working with a single type of ``$+ \infty$'' was sufficient).
\begin{prop}
  \label{prop:p-adic-types}
Let $\CM\prec\UU$ be a small model of $T$ and let $\alpha_1,\alpha_2\in \UU$ be
singletons with $\alpha_1\equiv_\emptyset
\alpha_2$ and  $v(\alpha_l)>v(m)$
for $l=1,2$ and every $m\in M$. Then $\alpha_1\equiv_M \alpha_2$.
\end{prop}
\begin{proof} Since $\alpha_1$ and $\alpha_2$ are singletons, by
  $P$-minimality, we need to show the following:
  \begin{enumerate}[itemsep=5pt]
  \item $p(\alpha_1)=0$ if and only if $p(\alpha_2)=0$ for any polynomial $p(x)\in
    M[x]$;

  \item  $\models V(p(\alpha_1)/q(\alpha_1))$ if and only if
    $\models V(p(\alpha_2)/q(\alpha_2))$, for  any $p(x),q(x)\in
    M[x]$;
  \item  $\models P_n(p(\alpha_1)/q(\alpha_1))$ if and only if
    $\models P_n(p(\alpha_2)/q(\alpha_2))$, for  any $n\geq 2$ and $p(x),q(x)\in
    M[x]$.
 \end{enumerate}

Now (1) holds since the assumption implies that both $\alpha_1$ and $\alpha_2$
are transcendental over $M$: if $p(\alpha_l) = 0$ for some $p(x) \in M[x]$, then, as $\CM$ is a model and $p(x) = 0$ is an $M$-definable algebraic set, necessarily $\alpha_l \in M$, but $v(\alpha_l) \not > v(\alpha_l) \in M$, contradicting the assumption. And (2) is equivalent to: 
\[ v(p(\alpha_1)) \geq v(q(\alpha_1)) \\\text{ if and
    only if }v(p(\alpha_2)) \geq v(q(\alpha_2)),
\]
for any non-zero $p(x),q(x)\in M[x]$. Let $p(x)=a_0+a_1x+\dotsc+a_kx^k$ and
$q(x)= b_0+b_1x+\dotsc+b_sx^s$. Let $i$ be minimal with $a_i\neq 0$ and $j$ be minimal with $b_j\neq
0$. Then, for $l=1,2$ we have
$v(p(\alpha_l))=v(a_i\alpha_l^i)=v(a_i)+iv(\alpha_l)$, 
and $v(q(\alpha_l))=v(b_j\alpha_l^j)=v(b_j)+jv(\alpha_l)$.
Thus $v(p(\alpha_l))\geq v(q(\alpha_l))$ if and only if $i>j$, or $i=j$ and
$v(a_i)\geq v(b_j)$. The latter condition is independent of $l$.

Finally, we demonstrate (3).  It is easy to see that $\models P_n(p(\alpha_l)/q(\alpha_l))$ if and only if
$\models P_n(p(\alpha_l)q^{n-1}(\alpha_l))$. Thus we need to show that
$\models P_n(p(\alpha_1))$ if and only if $\models P_n(p(\alpha_2))$, for any $p(x)\in M[x]$.

We will need the following fact that follows easily from henselianity of $\CM$ (which holds by Fact \ref{fac: P-minimal fields}).

\begin{fact}
  \label{fact:roots}
If $\varepsilon\in \UU$ satisfies $v(\varepsilon)>k$ for all $k\in
\NN$ then for any $n\in \NN$ the element $1+\varepsilon$ has $n$-th root. 
\end{fact}

Let $p(x)=a_0+a_1x+\dotsc a_kx_k$ be a nonzero polynomial over $M$ and
choose minimal $i$ such that $a_i\neq 0$.  Then, for $l=1,2$ we have
$p(\alpha_l)=a_i\alpha_l^i(1+\varepsilon_l)$ with $v(\varepsilon_l)>
k$ for all $k \in \mathbb{N}$,  and
$p(\alpha_l)$ has $n$-th root if and only if $a_i\alpha_l^i$ had $n$-th
root.  We can find $b_i\in M$ and $c\in \ZZ$ such that $a_i= b^nc$.
Hence $a_i\alpha_l^i$ has $n$-th root if and only if $c\alpha_l^i$
does. Since $\ZZ$ is in the definable closure of $\emptyset$, we have 
$\alpha_1 \equiv_\ZZ \alpha_2$  and $P_n(c\alpha_1)$ if and only if
$P_n(c\alpha_2)$.
\end{proof}

\begin{lem}\label{lem: poly bdd implies cofinal}
Let $\CM$ be a small model of $T$ and let $\alpha\in \UU$ satisfy
$v(\alpha)>v(m)$ for every $m\in M$. Then the sequence $\{ nv(\alpha) : 
n\in \NN \}$ is cofinal in the value group of $M\la \alpha\ra$, where
$M\la \alpha\ra$ is a prime model over $M\cup\{\alpha\}$ (i.e. $M\la \alpha\ra=\dcl(M\cup\{\alpha\})$). 
\end{lem}
\begin{proof}

Let $\gamma \in \Gamma_{M\la \alpha\ra}$ be arbitrary. As $\gamma \in \dcl(M \cup \{ \alpha \})$, we  have that $\gamma = v \left( f_m(\alpha) \right)$ for some $\emptyset$-definable family of functions $f$ and some tuple $m$ in $M$. Consider the $\emptyset$-definable family of functions $g: X \times \mathbb{M} \to \mathbb{M}$ with $X = \mathbb{M} \setminus \{0\}$ given by $g_m(x) := \frac{1}{f_m \left(\frac{1}{x} \right)}$.

Let $n \in \mathbb{N}$ and the definable map $a: M \to \Gamma_M$ be  given by Definition \ref{def: poly bdd P-min} for the family $g$ using Fact \ref{fac: P-min poly bdd}. Then $-a(m) \in \Gamma_M$, and so $v(\alpha)  > -a(m)$ by assumption. Hence $v\left(\frac{1}{\alpha}\right) = - v(\alpha) < a(m)$  and  so we have
$$ -v(f_m(\alpha)) = v \left(\frac{1}{f_m(\alpha)} \right)= v \left( g_m \left(\frac{1}{\alpha} \right) \right) > n v \left(\frac{1}{\alpha} \right) = -n v(\alpha).$$
Hence $\gamma = v(f_m(\alpha)) < nv(\alpha)$, as wanted.
\end{proof}

\begin{lem} \label{lem: p-adic growth condition}
Let $p \in S_1(\emptyset)$ be arbitrary. 
\begin{enumerate}
\item 	There is at most one  global type $\tilde{p} \in S_1(\UU)$ such that $\tilde{p} \supseteq p \cup \{ v(x) > v(m) : m \in \mathbb{M} \}$, and $\tilde{p}$ is $\emptyset$-invariant. 
\item Assume $\tilde p$ as in (1) exists.  	Let $\CM \prec \mathbb{M}$ and $\alpha_1,\dotsc,\alpha_n \in \mathbb{M}$.  Then $(\alpha_1,\dotsc \alpha_n)$
realizes $\tilde p^{(n)}|M$ if and only if each $\alpha_i$ realizes $p$, 
$v(\alpha_1)> v(m)$ for all $m\in M$ and $v(\alpha_{i+1})>kv(\alpha_i)$ for all $k\in \NN$ and $i=1,\dotsc,n-1$.

\end{enumerate}

\end{lem}

\begin{proof}
Part (1) follows from Proposition \ref{prop:p-adic-types} and $P$-minimality. 

Part (2) follows by the same argument as in Lemma \ref{lem: infty1 iff R-growing} using Lemma \ref{lem: poly bdd implies cofinal}.
\end{proof}

\begin{defn} For an integer $n>0\in \NN$, we say that a sequence $a_i,
  i=1,\dotsc, L$, of elements of $\CM_0$ is \emph{linearly $n$-growing} if
  $v(a_0)>n$ and $v(a_{i+1})> nv(a_i)$ for all $i$.
\end{defn}

  Notice that a subsequence of a linearly $n$-growing sequence is also
  linearly $n$-growing.

\begin{lem}\label{thm:value-grow}
 For any finite set of formulas $\Delta(x_1,\dotsc,x_k)$ with
  parameters from $\CM_0$ there are $n\in
  \NN$ and $d_0\in \NN$ such that any linearly $n$-growing sequence of elements
  $a_i\in \CM_0$ of length $N$ contains a $\Delta$-indiscernible
  subsequence of length at least $\frac{N}{d_0}$. 
\end{lem}
\begin{proof}

  Let $\Sigma(x_1,\dotsc,x_{2k})$ be the partial type that is the
  union of 
\[
    \Sigma_1= \left\{\bigwedge_{1\leq i<i\leq 2k} \varphi(x_i)\leftrightarrow
    \varphi(x_j) \colon \varphi(x) \text{ is an $\CL$-formula over $\emptyset$} \right\} \]
and 
\[ \Sigma_2=\left\{( x_1 >n )\wedge \bigwedge_{i=1}^{2k-1} v(x_{i+1}) >n v(x_i) \colon n\in \NN\right\}. 
\]

Let $( a_i : 1 \leq i \leq N )$ be any sequence of elements in $\UU$ such that all of the $a_i$'s have the same type 
 over the empty set. Assume that $v(a_0) > \NN$ and $v(a_{i+1})
> nv(a_i)$ for every $i=1,\dotsc,N-1$ and $n \in \NN$. Then, by Lemma \ref{lem: p-adic growth condition}, the sequence
$( a_i : 1 \leq i \leq N )$ realizes $\tilde{p}^{(N)}|\CM_0$ for $p=\tp(a_1/\emptyset)$, and so is indiscernible over $\CM_0$ by Fact \ref{fac: basic inv types}. It follows that 
\[
\Sigma(x_1,\dotsc,x_{2k})\vdash \psi(x_1,\dotsc,x_k)\leftrightarrow
\psi(x_{i_1},\dotsc,x_{i_k})  
\]
for any $1\leq i_1<i_2<\dotsc<i_k\leq 2k$ and $\psi \in \Delta$.

By compactness, there are finite subsets $\Sigma_1^0\subseteq \Sigma_1$ and $\Sigma_2^0\subseteq \Sigma_2$
such that 
\[
\Sigma_1^0\cup \Sigma_2^0\vdash  \psi(x_1,\dotsc,x_k)\leftrightarrow
\psi(x_{i_1},\dotsc,x_{i_k})  
\]
for any $1\leq i_1<i_2<\dotsc<i_k\leq 2k$ and $\psi \in \Delta$.

Let $\varphi_1,\dotsc,\varphi_s$ be all $\CL$-formulas over $\emptyset$ appearing
in $\Sigma_1^0$, and let $n\in \NN$ be maximal such that the condition $v(x_{i+1}) > nv(x_i)$ appears in  $\Sigma_2^0$.

Let $d_0=2^s$. Now any linearly $n$-growing sequence of length $N$ contains
a subsequence of length at least $\frac{N}{d_0}$ satisfying the same
$\varphi_1,\dotsc, \varphi_s$, and this subsequence is $\Delta$-indiscernible. 
\end{proof}

As in the $o$-minimal case, combining Lemma \ref{lem: finitary shrinking} with Lemma \ref{thm:value-grow} we can also allow additional parameters in $\Delta$.

\begin{cor}\label{thm:value-grow-uniform}
 For any finite set of $\CL$-formulas $\Delta(x_1,\dotsc,x_k; y)$ with
  parameters from $\CM_0$ there are $n\in
  \NN$ and $d\in \NN$ such that for all sufficiently  large $N$, for  any $ c\in \CM_0^{|y|}$, any linearly $n$-growing sequence 
  $\vec a= ( a_1,a_2,\dotsc, a_N)$ of elements from $\CM_0$   contains a
  $\Delta(x_1,\dotsc, x_k; c)$-indiscernible
  subsequence of length at least $\frac{N}{d}$. 
\end{cor}

\medskip

It remains to establish an analog of Fact \ref{prop:comb} in the $p$-adic case, demonstrating that there are ``enough'' linearly $n$-growing sequences.

For elements $\alpha \in \CM_0$ and $r$ in the value group of $\CM_0$, we will denote by
$B(\alpha,r)$ the closed ball in $\CM_0$ of (valuational) radius $r$ centered at $\alpha$,
i.e. 
\[ B(\alpha,r)=\{ a\in \CM_0 \colon v(a-\alpha)\geq r \}. \] 

\begin{rem}
Since the value group $\Gamma_{\CM_0} = \mathbb{Z}$ is discrete by assumption, and the residue field is $\mathbb{F}_q$ with $q=p^t$ for some prime $p$ and $t \in \mathbb{N}$ by Fact \ref{fac: P-minimal fields}, it is easy to see that every ball $B(\alpha,r)$ is given by a disjoint union of $q$ balls $B(\alpha_i, r+1)$, $i=0, \ldots, q-1$ with $\alpha_i = \alpha + i\beta $, where $\beta \in \CM_0$ is arbitrary with $v(\beta) = r$.
\end{rem}

\begin{lem}
  \label{lem:comb1}
Let $A \subseteq \CM_0$ be a finite non-empty set with $|A|\geq 2$
and $A\subseteq
B(\alpha,r)$ for some $\alpha\in \CM_0$ and $r\in \ZZ$. Then there is
$\alpha'\in B(\alpha,r)$ and $r' >r$ such that $ \frac{1}{q}|A|\leq  |A \cap B(\alpha',r')|<|A|$.
\end{lem}
\begin{proof}
  Let $r_1\in \ZZ$ be maximal such that some ball $B(\alpha_1,r_1)$ with $\alpha_1\in B(\alpha,r)$
  contains $A$ (so $r_1 \geq r$). As remarked above, the ball $B(\alpha_1,r_1)$ is the union of $q$ balls of
  radius $r'=r_1+1$. Hence for at least one of these $q$ balls, say $B(\alpha',r')$ with $\alpha'\in B(\alpha_1,r_1)$, we have $ \frac{1}{q}|A|\leq  |A \cap B(\alpha',r')|<|A|$ (the last inequality is by maximality of $r_1$).
\end{proof}

\begin{prop}\label{prop:diff-seq} Let $k\in \NN$ be positive. For every finite $A\subseteq \CM_0$ with $|A|\geq 2q^{k-1}$
  there is $\alpha\in \CM_0$ and elements $a_1,\dotsc a_k \in A$ such that
  the valuations of $\alpha-a_i,i=1,\dotsc,k$, are pairwise distinct.
\end{prop}
\begin{proof} Let $A\subseteq \CM_0$ be a finite set with $|A|\geq
  2q^{k-1}$. We set $A_0=A$, and also choose $\alpha_0\in \CM_0$ and
  $r_0\in \ZZ$ so that $A\subseteq B(\alpha_0,r_0)$. 
 
Using Lemma \ref{lem:comb1}, by  induction on $i=1,\dotsc,k$ we
construct  finite  sets $A_0
\supsetneqq A_1\supsetneqq \dotsb \supsetneqq A_k$, elements
$\alpha_1,\dotsc,\alpha_k\in \CM_0$ and integers $r_1< r_2 <\dotsc < r_k$
such that:  
\begin{itemize}
\item $A_i=B(\alpha_i,r_i)\cap A$;
\item $|A_i|\geq \frac{1}{q^i}|A|$;
\item $\alpha_{i} \in B(\alpha_{i-1}, r_{i-1})$ for $i=1,\dotsc, k$.
\end{itemize}

We take $\alpha=\alpha_k$, and for $i=1,\dotsc,k$ we let $a_i\in
A_{i-1}\setminus A_i$ be arbitrary. Then $a_i  \in B(\alpha,r_{i-1}) \setminus
B(\alpha,r_i)$, hence $  r_{i-1} \leq v(\alpha-a_i) < r_i$.
\end{proof}

We will use the following combinatorial facts.

\begin{fact}(Erd\H{o}s--Szekeres Theorem  \cite{erdios1935combinatorial})\label{fac: Erdos-Szekeres}
For any $r, s \in \mathbb{N}$, any sequence of pairwise distinct real numbers of length at least $(r-1)(s-1)+1$ contains an increasing subsequence of length $r$ or a decreasing subsequence of length $s$.
\end{fact}

\begin{fact}\cite[Lemma 4.1]{bukh2014erdHos}\label{fac: doubling subsequence}
Given $n \in \mathbb{N}$, every strictly increasing sequence of real numbers of length $4^n$ contains a subsequence $(b_1, \ldots, b_n)$ of length $n$ such that either $b_2 - b_1 \geq 2$ and $b_{i+1} - b_1 \geq 2 (b_i - b_1)$ for all $i = 2, \ldots, n-1$, or $b_n - b_{n-1} \geq 2$ and $b_{n} - b_{n - (i+1)} \geq  2( b_n - b_{n - i})$ for all $i = 1, \ldots, n-2$.
\end{fact}

\begin{prop}
  \label{prop:main p-adic} 
There are finitely many functions $F_1(x,\bar
y),\dotsc, F_s(x,\bar y)$ definable with parameters from $\CM_0$ such that for any $n\in \NN$ there is a
constant $C > 0$ such that for   any $k\in
\NN$ the following holds. For any $K\geq 2^{2^{Ck}}$ and any sequence $\vec a= ( a_1,\dotsc, a_K )$ in $\CM_0$ there are a
linearly $n$-growing sequence $\vec b= ( b_1,\dotsc, b_k )$ of
elements  in $\CM_0$, $\bar c\in
\CM_0^{|\bar y|}$ and $i\in \{ 1,\dotsc,s \}$ such that one of the sequences 
\[F_i(\vec b, \bar c) := \left(  F_i(b_1, \bar c), F_i(b_2,\bar c),\dotsc,       F_i(b_k,\bar
c) \right) \] 
or 
\[F_i(\, \cev b, \bar c) := \left( F_i(b_k,\bar c), F_i(b_{k-1},\bar c),\dotsc,       F_i(b_1,\bar
c)\right) \] 
is a subsequence of $\vec a$.

\end{prop}
\begin{proof} 
\newcommand{\cell}[1]{\lceil #1 \rceil}
As usual, for a real number $R$ we will denote by $\cell R$
the smallest integer $N$ satisfying $N\geq R$.

First, notice that it is sufficient to prove the proposition for
$n=2$. Indeed if $a_1, a_2,\dotsc, $ is a linearly $2$-growing sequence of length $N$,
then for a given $n$, taking $l=\cell {\log_2{n}}$, the sequence $a_l,
a_{2l},\dotsc$ is a linearly $n$-growing sequence of length at least $\frac{N}{l}$.  

\medskip 
Assume $k$ is given and  $K\geq 2^{2^{Ck}}$, where a suitable constant $C$ will be
  determined in the proof.  
Let $a_1,\dotsc a_K$ be a sequence of elements of $\CM_0$. 

\medskip
\noindent{\bf Case 1}.  The sequence $a_1,\dotsc a_K$ contains
at
least $\sqrt{K}$ equal elements.  

Let's call this repeated element $a'$. Then the conclusion of the proposition holds as we can map any linearly
$2$-growing sequence of length $\cell{\sqrt K}$ onto a subsequence $\vec a$ using
the constant map $F(x,a') := a'$.

\noindent{\bf Case 2}.  The sequence $a_1,\dotsc a_K$ does not
contain $\cell{\sqrt{K}}$ equal elements. Then it contains at least
$K_1 : =\sqrt{K}$ pairwise distinct elements.

Using  Proposition
\ref{prop:diff-seq} we can find an element $\alpha\in \CM_0$ and a
subsequence $\vec a^1= ( a^1_1,\dotsc a^1_{K_2} ) $ of $\vec a$ with 
$K_2 := \cell{ \log_q(\frac{1}{2} K_1)+1 }$ such that the valuations
$v(\alpha-a^1_i)$ are pairwise distinct for all $1 \leq i \leq K_2$. 

Thus, using the map $F(x,\alpha) = x+\alpha$ we can find a sequence 
$\vec b= ( b_1,\dotsc b_{K_2} )$ such that $F(\vec b,\alpha) = \left( F(b_1, \alpha), \ldots, F(b_{K_2}, \alpha) \right)$ is a
subsequence of $\vec a$ and  all of the valuations $v(b_i)$ are pairwise
distinct. 

By the Erd\H{o}s--Szekeres Theorem (Fact \ref{fac: Erdos-Szekeres}),  the sequence $\vec b$ contains a
subsequence $\vec b^1= ( b^1_1,\dotsc, b^1_{K_3} )$ with $K_3 := 
\cell{\sqrt{K_2}}$ such that the corresponding sequence of valuations
$( v(b^1_1),\dotsc, v(b^1_{K_3}) )$ is either increasing or
decreasing.  Using the function $F(x)=x^{-1}$ if needed, we can assume that the
sequence is increasing.

By Fact \ref{fac: doubling subsequence}, there is a subsequence 
$\vec b^2 = ( b^2_1,\dotsc, b^2_{K_4} )$ of $\vec b^1   $ with $K_4 := \cell{ \frac{1}{2} \log_2 K_3 }$ such
that either for  the sequence 
\[ \vec v=\left( v(b^2_2)-v(b^2_1), v(b^2_3)-v(b^2_1), \dotsc,
v(b^2_{K_4})-v(b^2_1) \right) \]
we have $v_1\geq 2, v_{i+1}\geq 2v_i$, or for the sequence  

\[ \vec v'=\left( v(b^2_{K_4})-v(b^2_{K_4-1}), v(b^2_{K_4})-v(b^2_{K_4-2}), \dotsc,
v(b^2_{K_4})-v(b^2_1) \right) \]
we have  $v_1\geq 2, v_{i+1}\geq 2v_i$.
  
In the first case, the sequence 
$\vec b^3= ( b^2_2/b^2_1,\dotsc, b^2_{K_4}/b^2_1 )$ is linearly
$2$-growing and can be embedded into $\vec b^2$ via the transformation
$x\mapsto b^2_1x$. 

In the second case, the sequence 
 $\vec b^4= ( b^2_{K_4}/b^2_{K_4-1},\dotsc, b^2_{K_4}/b^2_1 )$ is linearly
$2$-growing, and its reverse sequence $\cev{b}^4$  can be embedded into
$\vec b^2$ via the transformation $x\mapsto x^{-1}/ b^2_{K_4}$. 

Hence we have demonstrated that every sequence of length $K$ contains a subsequence of length $\min \{\cell{\sqrt K}, K_4 \}$ with the desired property. Going backwards through the proof we have $K_3 \leq 4^{K_4}$, $K_2 \leq K_3^2$, $K_1 \leq 2 q^{K_2-1}$ and $K \leq K_1^2$. As $q$ is fixed, an easy calculation shows that for any sufficiently large $C \in \mathbb{R}$, we have $K \leq 2^{2^{C K_4}}$ for all sufficiently large $K_4 \in \mathbb{N}$, and taking $K_4 = k$ we can conclude the result.
\end{proof}

Combining Corollary \ref{thm:value-grow-uniform} and Proposition \ref{prop:main p-adic} exactly as in the $o$-minimal case (see Theorem \ref{thm:main}), we obtain Theorem \ref{thm: BM in P-min gen}.

\section{Ramsey growth in NIP} \label{sec: Ramsey growth NIP}
In this section we consider Ramsey numbers for definable relations of higher arity.
We fix a structure $\CM$ in a language $\CL$, and by a ``formula''  we always mean an $\CL$-formula. Following the method of \cite{conlon2014ramsey} for the semialgebraic case, we obtain the following recursive bound for higher arity Ramsey numbers in arbitrary NIP structures.
\begin{thm}\label{thm: stepping down NIP}
Let $\CM$ be an NIP structure, $k\geq 3$ and $\varphi(x_1, \ldots, x_k;z)$ a formula
with $|x_1|=\dotsb=|x_k|=d$. Then, defining the formula $\psi(x_1,\dotsc,x_{k-1};z') := \varphi(x_1,\dotsc,x_{k-1}; x_k,z)$ and taking $m := R^{*}_\psi(n-1) $, for all large enough $n$ we have $$R^*_\varphi(n) \leq 2^{C m \log m}$$ for some constant $C = C(\varphi)$.
\end{thm}
\begin{proof}
We are generalizing the argument from \cite[Theorem 2.2]{conlon2014ramsey}.

Let $e\in M^{|z|}$ be arbitrary, 
 $\psi(x_1,\dotsc,x_{k-1};z')=\varphi(x_1,\dotsc,x_{k-1}; x_k,z)$, $n\in \NN$ large enough (to be determined in the proof), and $m= R^{*}_\psi(n-1)$. Let $\vec a = (a_1, \ldots, a_N )$ be a sequence of elements in $M^d$  with $N \geq  2^{C m \log m}$, where $C=C(\varphi)$ is a constant to be specified later. We need to find a $\varphi(x_1,\dotsc,x_k;e)$-indiscernible subsequence of $\vec a$ of length $n$.  

Let $E\subseteq (M^d)^k$ be the $k$-ary relation on $M^d$ defined by $\varphi(x_1,\dotsc,x_k;e)$, i.e.
$E  =\{(t_1, \ldots, t_k) \in (M^{d})^{k} : \CM \models \varphi(t_1, \ldots, t_k; e)\}$.

The idea is to find  a subsequence $\vec b=(b_1,\dotsc, b_{m+1})$ of $\vec a$ such that 
for all $1\leq i_1<\dotsc < i_{k-1} \leq m$, either $(b_{i_1}, \dotsc,b_{i_{k-1}}, b_i)\in E$ for all $i_{k-1}<i\leq m+1$ or
 $(b_{i_1}, \dotsc,b_{i_{k-1}}, b_i)\not\in E$ for all $i_{k-1}<i\leq m+1$.

To build a sequence $\vec b$ as above,  
we recursively choose elements $b_r$ in $\vec a$ and also subsequences  $\vec c_r$ of $\vec a$  for $r = k-2, k-1, \ldots, m+1$ with $\vec c_{r+1} \subset \vec c_r$
 so  that the following holds.
\begin{enumerate}
\item For every $(k-1)$-subsequence  $(b_{i_1}, \ldots, b_{i_{k-1}})$ of  $(b_1, \ldots, b_{r-1} )$ with $i_1 < \ldots < i_{k-1}$, either $(b_{i_1}, \ldots, b_{i_{k-1}}, b) \in E$ for every $b \in \{ b_j : i_{k-1} < j \leq r \} \cup \vec c_r$ or $(b_{i_1}, \ldots, b_{i_{k-1}}, b) \notin E$ for every $b \in \{ b_j : i_{k-1} < j \leq r \} \cup \vec c_r$.
\item $|\vec c_r| \geq \frac{N}{C_1^r r^{C_2 r}}$, where $C_1, C_2$ are some constants depending just on $\varphi$.
\item The subsequence $(b_1,\dotsc,b_r)$ appears in $\vec a$ in front of the subsequence $\vec c_r$, i.e.
$(b_1,\dotsc,b_r)\hat\ \vec c_r$ is a subsequence of $\vec a$. 
 \end{enumerate}

We start with $r=k-2$ by taking  $(b_1, \ldots, b_{k-2}) = ( a_1, \ldots, a_{k-2} )$ and  $\vec c_{k-2} = 
(a_{k-1},\dotsc,a_N)$. 
Assume we have obtained $( b_1, \ldots, b_r )$ and $\vec c_r$ satisfying (1)--(3) above, and we define $b_{r+1}$ and $\vec c_{r+1}$ as follows. 

Let $b_{r+1}$ be the first element in $\vec c_r$ and let $\vec c_r^{\,*}$ be the sequence $\vec c_r$ with the first element removed.   Let $\theta(x_k; u)$ be the partitioned formula obtained from 
 $\varphi(x_1\ldots, x_{k-1}, x_k, z)$ by partitioning its variables into two groups $x_k$ and $u = x_1, \ldots, x_{k-1},z$.
As the formula $\theta$ is NIP, by Fact \ref{fac: PolyTypesNIP} the number of complete $\theta(x_k; u)$-types over an arbitrary finite set $D\subseteq M^{|z|+(k-1)d}$ of parameters is bounded by $C_3|D|^{C_4}$ for some constants $C_3,C_4$ depending just on $\varphi$.

Let $D=\{ (b_{i_1},\dotsc,b_{i_{k-1}},e) \colon 1\leq i_1 <\dotsc <i_{k-1}\leq r+1\}$. 
Obviously $|D|\leq (r+1)^{k-1}.$

It follows by the pigeonhole principle that there is some complete $\theta$-type $p(x_k) \in S_\theta(D)$ such that
the number of elements in $\vec c_r^{\,*}$ realizing $p(x)$ is at least 
$\frac{|{\vec c}_r^{\,*}|}{C_3|D|^{C_4}}\geq  \frac{|{\vec c}_r|-1}{C_3|(r+1)|^{(k-1)C_4}}\geq 
\frac{|{\vec c}_r|}{2C_3|(r+1)|^{(k-1)C_4}}$, provided $|\vec c_r| \geq 2$.  

We take $\vec c_{r+1}$ to be the subsequence of elements of $\vec c_r$ realizing $p$. For $C_1=2C_3$ and $C_2=(k-1)C_4$ (again, both $C_1$ and $C_2$ only depend on $\varphi$), using the inductive lower bound for the length  of $\vec c_r$ and calculating, we obtain $|\vec c_{r+1}| \geq \frac{N}{C_1^{(r+1)}( r+1) ^ {C_2 (r+1)}}$, i.e. (2) is satisfied.

Now for any subsequence  $(b_{i_1}, \ldots, b_{i_{k-1}})$ of  $( b_1, \ldots, b_{r+1} )$, we have that either $(b_{i_1}, \ldots, b_{i_{k-1}},b) \in E$ for all $b \in \vec c_{r+1}$ or $(b_{i_1}, \ldots, b_{i_{k-1}},b) \notin E$ for all $b \in \vec c_{r+1}$. Together with the inductive assumption this implies that (1) is satisfied by $(b_1, \ldots, b_{r+1})$ and $\vec c_{r+1}$.
Finally, (3) is clear from the construction.

For $\vec c_m$ to be non-empty (in which case we would have constructed our  sequence $(b_1, \ldots, b_{m+1})$), by (2) we need  $\frac{N}{C_1^{m} m ^ {C_2 m}}\geq 1$, i.e. 
$N\geq C_1^mm^{C_2m}$.  It is not hard to find a constant $C$, depending on $C_1, C_2$ only, so that the condition $N\geq 2^{Cm\log m}$ is sufficient.

\end{proof}

\begin{rem}
The constant $C_4$ in the above proof depends just on the VC-density of $\varphi$ (with a corresponding partition of the variables). By Fact \ref{fac: bounds on UDTFS}, in the case of $o$-minimal theories we can take $C_4 = d$.
\end{rem}

By a repeated application of Theorem \ref{thm: stepping down NIP} we have an improved bound on Ramsey numbers for relations of higher arities.

\begin{thm}\label{thm: repeated tower NIP}
Let $\mathcal{M}$ be an NIP structure, and assume that for all $\psi(x_1,x_2;z)$  we have $R^*_\varphi(n) \leq n^c$ for some $c = c(\psi)$ and all $n$ large enough. Then for all $\varphi(x_1,\dotsc,x_k;z')$ we have $R^*_\varphi(n) \leq \twr_{k-1}(n^c)$ for some $c = c(\varphi)$ and all $n$ large enough.
\end{thm}

Now we discuss the connection of the assumption of Theorem \ref{thm: repeated tower NIP} with the (strong) Erd\H{o}s-Hajnal property for graphs definable in $\CM$.

\begin{defn}\cite{FoxPach} \label{def: strongEH}
\begin{enumerate}
	\item Let $\mathcal{G}$ be a class of finite graphs (i.e.~the edge relation is assumed to be symmetric and irreflexive). 
We say that $\mathcal{G}$ has the \emph{Erd\H{o}s-Hajnal property}, or the \emph{EH property}, if
there is $\delta>0$ such that 
every $G=(V,E)\in \mathcal{G}$ has a homogeneous subset $V_0$ of size $|V_0| \geq |V|^\delta$ (i.e.~either $(a,b) \in E$ for all $a \neq b \in V_0$, or $(a,b) \notin E$ for all $a \neq b \in V_0$).

\item Let $\mathcal{G}$ be a class of finite binary relations, i.e.~ every member of $\CG$ is of the form $(E,V_1,V_2)$, where $E \subseteq V_1 \times V_2$ with $V_1,V_2$ finite sets (not necessarily disjoint). We say that $\CG$ has the \emph{strong EH property} if
there is $\delta>0$ such that 
for every $(E, V_1, V_2) \in \mathcal{G}$ there are subsets $V'_i \subseteq V_i$ with
$|V'_i| \geq \delta|V_i|$ for $i=1,2$ such that the pair of sets $V'_1, V'_2$ is homogeneous (i.e.~either $V'_1 \times V'_2 \subseteq E$ or $V'_1 \times V'_2 \cap E = \emptyset$).
\item A family of finite graphs $\CG$ has the strong EH property if the family of finite binary relations $\{ (E,V,V) : (E,V) \in \CG \}$ has the strong EH property.

\end{enumerate}
	
\end{defn}

We recall that a famous conjecture of Erd\H{o}s and Hajnal \cite{chudnovsky2014erdos} says that for every finite graph $H$, the family of all finite graphs not containing an induced copy of $H$ has the EH  property.

  \begin{defn} \label{def: definable EH}Let $\CM$ be a first-order structure, and $\varphi(x_1, x_2; z)$ a formula with $|x_1|=|x_2| = d$. Let $\mathcal{G}_\varphi$ be the family of all finite binary relations $(E,V_1,V_2)$ with 
$V_1, V_2 \subseteq M^d$ finite and $E=(V_1 \times V_2)\cap \varphi(M,b)$ for some $b \in M^{|z|}$. Let $\CG^{\textrm{sym}}_\varphi$ be the family of all finite graphs $(V,E)$ with $V \subseteq M^d$ and $E = (V \times V) \cap \varphi(M,b) \setminus \Delta$ for some $b \in M^z$ such that $E$ is symmetric (where $\Delta = \{ (v,v) : v \in V \}$ is the diagonal).
We say that $\varphi$ satisfies the EH property (respectively  strong EH property) if the family $\mathcal{G}^{\textrm{sym}}_\varphi$ (respectively $\CG_\varphi$) does.

If this holds for all formulas $\varphi$ in $\CM$, we say that $\CM$ satisfies the (strong) EH property.
\end{defn}

\begin{rem} \label{rem: strong EH implies EH} It is shown in \cite{alon2005crossing} that if a family of finite graphs $\mathcal{G}$ has the strong EH property and is closed
  under taking induced subgraphs then it has the EH property. In particular, this applies to every family of the form $\CG^{\textrm{sym}}_\varphi$ as in  Definition \ref{def: definable EH}.
\end{rem}

Hence, $\CM$ satisfies the EH property precisely when the assumption of Theorem \ref{thm: repeated tower NIP} holds for all \emph{symmetric} definable relations. By the results in \cite{chernikov2015regularity} we know that this property holds in arbitrary reducts of distal structures.
\begin{defn}\label{def: distality}A structure $\CM$ is \emph{distal} if the following holds.

 For every formula $\varphi(x,y)$ there is a formula $\theta(x,y_1,\dotsc,y_n)$ with $|y_1|=\dotsb=|y_n|=|y|$  such that: for any finite $B \subseteq M^{|y|}$ with $|B| \geq 2$ and any $a \in M^{|x|}$, there are  $b_1,\dotsc,b_n \in B$ such that $\CM \models \theta(a,b_1,\dotsc,b_n)$ and for any $b \in B$, 
either $\varphi(M,b)\subseteq \theta(M,b_1,\dotsc,b_n)$ or $\varphi(M,b)\cap \theta(M,b_1,\dotsc,b_n)=\emptyset$.

\end{defn}
Distality was introduced in \cite{DistalPierre}, the equivalence of the original definition and the combinatorial definition above is from \cite{chernikov2015externally}, and the connection to combinatorics is from \cite{chernikov2015regularity} (see also \cite{DistalCutting}). Important examples of distal structures are given by arbitrary (weakly) $o$-minimal  and $P$-minimal structures. We refer to the introduction of \cite{chernikov2015regularity} for a detailed discussion of distality.

\begin{fact}\label{fact:distal-SEH-implies-EH}\cite{chernikov2015externally} If $\CM$ is a reduct of a distal structure then it satisfies the strong EH property (and so the EH property as well, by Remark \ref{rem: strong EH implies EH}).
	
\end{fact}

In the next proposition, we demonstrate that in any structure satisfying the strong EH property, all (not necessarily symmetric) definable binary relations also  satisfy a polynomial Ramsey bound.

\begin{prop}\label{prop: sym EH implies EH}
Let $\CM$ be a structure satisfying the strong EH property. Then for all $\varphi(x,y;z)$ with $|x|=|y|$ we have $R^*_\varphi(n) \leq n^c$ for some $c = c(\varphi)$ and all $n$ large enough.
\end{prop}
\begin{proof}
	
Let $d:=|x|=|y|$, and let $E(x,y) \subseteq M^{d} \times M^{d}$ be a definable relation given by $\varphi(x,y;b)$ for some parameter $b \in M^{|z|}$.
We want to show that there is some real $c = c(\varphi) > 0$ such that every finite sequence from $M^{d}$ of length
$n$ contains an $E$-indiscernible subsequence of length $n^c$.  By  Fact~\ref{fact:distal-SEH-implies-EH} we know that it is true in the case when $E$ is
symmetric.

Let $\vec a=(a_1,\dotsc, a_n)$ be a sequence in $M^d$. For simplicity we will
assume that all of the $a_i$'s are pairwise distinct. We can always achieve it by
taking a subsequence of length $\sqrt{n}$. We will also assume that
$\models \neg E(x,x)$, i.e. $E$ is irreflexive (replacing $E(x,y)$ by $E'(x,y) := E(x,y) \land x \neq y$, any $E'$-indiscernible subsequence of $\vec{a}$ is also $E$-indiscernible).

Consider the relation $E_0(x,y)= E(x,y)\vee E(y,x)$. 
It is symmetric. Hence $\vec a$ contains an $E_0$-indiscernible
subsequence of length $n^{c_1}$, with $c_1 = c_1(\varphi) > 0$. If $\neg E_0(x,y)$ holds on every increasing pair of elements in this 
subsequence  then we are done. Otherwise, replacing $\vec a$ with this subsequence, we may assume that 
$E_0(x,y)$ holds on $\vec a$.

Now consider the relation $E_1(x,y)= E(x,y)\wedge E(y,x)$. Again it is
symmetric, so $\vec a$ contains an $E_1$-homogeneous subsequence of
polynomial length.  If $E_1(x,y)$ holds on this subsequence then we
are done. Assume otherwise, then again replacing $\vec a$ with this subsequence
we may assume that $\neg E_1(x,y)$ holds on $\vec a$. 
\medskip

Let $A=\{ a_i \colon i\in \{ 1, \ldots, n \}\}$.  We have that for $a\neq b\in A$
exactly one of $E(a,b)$ or $E(b,a)$ holds, and we also have that  $\neg
E(a,a)$ holds for all $a\in A$.  Hence $E$ is a \emph{tournament} on $A$.

Our goal is to show that for some $A_0\subseteq A$ of size $n^{c_2}$, with $c_2 = c_2(\varphi) > 0$, $E$ restricted to $A_0$ defines a linear order.  Then, by the 
Erd\H{o}s-Szekeres Theorem (Fact \ref{fac: Erdos-Szekeres}), a subsequence corresponding to $A_0$ would contain an
$E$-monotone subsequence of length $\sqrt{|A_0|}$ and we would be  done.

For an integer $m\leq n$, let's denote by $f(m)$ the
maximal $k$ such that every subset $A'\subseteq A$ of size $m$
contains a linearly ordered subset of size $k$. Obviously, we have $f(m)\geq 1$ for all $m\geq 1$. 

Now we use the strong EH property. We know that there is $0< \alpha <1$, with $\alpha = \alpha(\varphi)$, such
that for any $B\subseteq A$ there are disjoint subsets $B_0,B_1 \subseteq B$
with $|B_0|,|B_1|\geq \alpha|B|$ that are $E$-homogeneous. 
If $C_0\subseteq B_0, C_1\subseteq B_1$ are subsets linearly ordered
by $E$, then by $E$-homogeneity   $C_0 \cup C_1$ is also linearly ordered
by $E$.

This implies that $f(m)\geq 2f(\alpha m)$ and for any $s\in \NN$ we get 
$f(m)\geq 2^sf(\alpha^s m)$.

Recall that $|A|=n$. We choose the maximal $s$ such that $\alpha^s n \geq
1$.  Up to taking the integer part, we have 
\[ s\log(\alpha)+\log(n)\geq 0, \text{ i.e. we have } 
s \geq \frac{-\log{n}}{\log{\alpha}}.
 \] 
Then we get 
\[ f(n)\geq 2^\frac{-\log{n}}{\log{\alpha}}=n^{-\frac{1}{\log{\alpha}}}, \]
and taking $c := - \frac{1}{2 \log \alpha} > 0$  we can conclude the result.
\end{proof}

Hence the assumption of Theorem \ref{thm: repeated tower NIP} is satisfied in reducts of distal structures by Fact \ref{fac: EH in distal} and Proposition \ref{prop: sym EH implies EH}. As every distal structure is NIP, and every reduct of an NIP structure is NIP, applying Theorem \ref{thm: stepping down NIP} we get the following.

\begin{cor}
Let $\CM$ be a reduct of a distal structure. Then for any $\varphi(x_1,\dotsc,x_k;z)$   we have $R^*_\varphi (n) \leq \twr_{k-1}(n^c)$ for some $c = c(\varphi)$ and all $n$ large enough.
\end{cor}

Note that the assumption of Theorem \ref{thm: repeated tower NIP} is also trivially satisfied in stable structures by Fact \ref{fac: stable EH}. We conjecture that it holds in arbitrary NIP structures.

\begin{conj}
If $\CM$ is an NIP structure and $\varphi(x_1,x_2;z)$ is a formula, then $R^*_\varphi(n) \leq n^c$ for some $c = c(\varphi, \CM)$ and all sufficiently large $n$.
\end{conj}

This conjecture, in the case of a symmetric formula, is equivalent to saying that all graphs definable in NIP structures satisfy the Erd\H{o}s-Hajnal property. We refer the reader to \cite{chernikov2015regularity, chernikov2016definable} for further discussion.

\bibliographystyle{acm}
\bibliography{refs}
\end{document}